\documentclass[10pt,a4paper]{article}
\usepackage[latin1]{inputenc}
\usepackage{amsmath}
\usepackage{amsfonts}
\usepackage{amssymb}
\usepackage{graphicx}
\usepackage{caption}
\usepackage{subcaption}
\usepackage{fullpage}
\usepackage{color}

\usepackage{authblk}
\usepackage{pgfplots}
\pgfplotsset{compat=newest}
\pgfplotsset{plot coordinates/math parser=false}

\usepackage[numbers]{natbib}

\newcommand{\cboxxxs}[1]{\fboxsep0.5pt \colorbox{white}{#1}}
\newcommand{\cboxs}[1]{\fboxsep1.0pt \colorbox{white}{#1}}

\usepackage{tikz}

\usetikzlibrary{calc,patterns,plotmarks,decorations.pathmorphing,	decorations.markings,arrows}

\usepackage{enumerate}

\usepackage{amsthm}
\newtheorem{theorem}{\bf Theorem}[section]
\newtheorem{lemma}{\bf Lemma}[section]

\newtheorem{corollary}{\bf Corollary}[section]
\theoremstyle{definition}
\newtheorem{definition}{ \bf Definition}[section]
\theoremstyle{remark}
\newtheorem{remark}{\bf Remark }[section]

\begin{document}
	\title{Explicit Backbone Curves from Spectral Submanifolds of Forced-Damped Nonlinear Mechanical Systems}
	\author{Thomas Breunung}
	\author{George Haller \footnote{Corresponding author. Email: georgehaller@ethz.ch } }
	\affil{Institute for Mechanical Systems, ETH Z{\"u}rich, \protect\\  Leonardstrasse 21, 8092 Z{\"u}rich, Switzerland}
	\maketitle
\begin{abstract}
	Spectral submanifolds (SSMs) have recently been shown to provide exact and unique reduced-order models for nonlinear unforced mechanical vibrations. Here we extend these results to periodically or quasiperiodically forced mechanical systems, obtaining analytic expressions for forced responses and backbone curves on modal (i.e. two-dimensional) time dependent SSMs. A judicious choice of the parameterization of these SSMs allows us to simplify the reduced dynamics considerably. We demonstrate our analytical formulae on three numerical examples and compare them to results obtained from available normal form methods. 
\end{abstract}

	\section{Introduction}
		In drawing conclusions about a nonlinear mechanical system, an engineering analyst usually faces the challenge of high dimensionality and complex dynamic equations. To reduce simulation time and deduce general statements, it is desirable to reduce the dimension of the system and simplify the resulting reduced equations of motion.

	For linear systems, decomposition into normal modes is a powerful tool to derive reduced-order models. While the lack of the superposition principle makes such a decomposition impossible for nonlinear systems, various definitions of nonlinear normal modes are also available in the literature~(cf.~Rosenberg~\cite{ROSENBERG_NNM}, Shaw and Pierre~\cite{SHAW+Pierre} and Haller and Ponsioen~\cite{article:Haller_SSMbasics}). Specifically, Rosenberg~\cite{ROSENBERG_NNM} defines a nonlinear normal mode as a synchronous periodic orbit of a conservative system. Later Shaw and Pierre~\cite{SHAW+Pierre} extended this definition to dissipative systems, by viewing a nonlinear normal mode as an invariant manifold tangent to a modal subspace of an equilibrium point.  Sought in practice via a Taylor expansion, these manifolds serve as nonlinear  continuations of the invariant  modal subspaces spanned by the  eigenvectors of the linearized system. Due to their invariance, these manifold are natural candidates for model order reduction. 

While there are generally infinitely many Shaw-Pierre type surfaces for each modal subspace (cf. Neild~et~al.~\cite{Neild_2ndNF_overview}), Haller and Ponsioen~\cite{article:Haller_SSMbasics} have shown that, under appropriate non-resonance conditions, there is a unique smoothest one, which they called a spectral submanifold (SSM). When the underlying modal subspace is the one with the slowest decay, the dynamics on its corresponding SSM serves as the optimal, mathematically exact reduced-order model for the system dynamics (see Haller and Ponsioen~\cite{article:Haller_SSMbasics}).  Applications of this model reduction approach appear in Jain~et~al.~\cite{Jain_SFD_and_SSM} and  Szalai~et~al.~\cite{article:SSM_SysID}. Ponsioen et al.~\cite{Sten_auto_SSM} provide an automated computation package for two-dimensional SSMs of a general autonomous, nonlinear mechanical system.

While most of the above work focuses on unforced (autonomous) mechanical systems, here we explore further the utility of SSMs for forced dissipative nonlinear mechanical systems. For this class of systems, the existence, uniqueness and regularity of SSMs has been clarified by Haller and Ponsioen~\cite{article:Haller_SSMbasics}, relying on the more abstract invariant manifold results of Haro and de la Lave~\cite{HARO_SSM}. In this context, a nonlinear normal mode (NNM) is defined as the continuation of the trivial hyperbolic fixed point of the time-independent system under the addition of small time-dependent forcing with a finite number of frequencies. Depending on the frequency content of the time-varying terms, this continuation is a periodic or quasi-periodic orbit~(cf.~Haller~and~Ponsioen~\cite{article:Haller_SSMbasics}). The SSM will be  a time-dependent surface with the same frequency basis.  This SSM is then tangent to the NNM along directions associated with a spectral subspace of the linearization. 

The first attempt to construct such a non-autonomous SSM can be found in Jiang et al.~\cite{Jiang_NNM}, who formally reduce an externally forced, dissipative mechanical system to a two-dimensional time-varying invariant manifold. While their results are promising even for high amplitude oscillations, they are only able to carry out the reduction numerically for fixed parameter values, aided by a Galerkin projection. Therefore, their study is limited to specific examples and symbolic equations from which general conclusions about the forced response could be derived,  are not obtained. Furthermore, the uniqueness, existence and smoothness of their assumed invariant manifold remains unclear from their procedure.  

Extending this approach to systems with time-periodic coefficients in their linear part, Sinha et al.~\cite{sinha2005order} and Gabale and Sinha~\cite{Gabale_SSM} expand the assumed invariant manifold in a multivariate Taylor-Fourier series, obtaining the unknown coefficients from the invariance of the manifold. With unclear uniqueness, existence and smoothness properties of the manifold, however, the series expansion remains unjustified. Furthermore, the approach does not yield generally applicable closed formulas and hence numerical integration is required to analyze the reduced model.

A generally applicable procedure for the simplification of the (formally) reduced dynamics is the method of normal forms (cf. e.g. Guckenheimer and Holmes~\cite{guckenheimer_NL}). The method applies a series of smooth transformations to obtain a Taylor series of the original dynamical equations, which contain only the terms essential for the dynamics. Jezequel and Lamarque~\cite{JEZEQUEL} demonstrate the potential of normal forms for mechanical vibrations after the system is transformed to first-order phase-space form. Neild and Wagg~\cite{Neild_2NF_firstresult} give an alternative formulation of the normal form procedure that is directly applicable to second-order mechanical systems. Since all state variables are transformed, the resulting dynamics have the same dimensionality as the original system and no model-order reduction is achieved. Furthermore, both of these normal form approaches start from conservative systems and treat damping as a small bifurcation parameter. Therefore, the unfolding from the conservative limit has to be discussed for every damping type separately.

Touz{\'e} and Amabili~\cite{touze_damped_NNM} seek to unite normal form theory with model-order reduction for the first time. After a normal form transformation, they restrict their calculations to heuristically chosen submanifolds. As pointed out by the authors, a strict time-varying normal form is not computed. Instead, the forcing is inserted directly into the normal form. This represents phenomenological forcing aligned with a curvilinear coordinates, rather than specific physical forcing applied to the system.

Due to the essential nonlinear relationship between forcing and response amplitude of nonlinear systems, a single response curve for a given forcing is  meaningless for different forcing amplitudes. To summarize responses obtained from different forcing amplitudes, one may choose to collect distinguished points of various response curves in the same diagram. For, instance,  Nayfeh and Mook~\cite{nayfeh_nonlinear} and Cveticanin~et al.~\cite{cveticanin2017dynamics} call the curve formed by the loci of the maximal response amplitude the backbone curve. Cveticanin~et al.~\cite{cveticanin2017dynamics} further trace fold points of the forced response and relate them to the maximum amplitude. Both Nayfeh and Mook~\cite{nayfeh_nonlinear} and Cveticanin~et al.~\cite{cveticanin2017dynamics}, however compute the backbones curves only for low dimensional specific examples. Furthermore, Peeters~et~al.~\cite{peeters_backbone} trace the frequencies at which the forced response is 90 degree out of phase to the forcing. 

An alternative given by Klotter~\cite{Klotter} and continued by Rosenberg and Atkinson~\cite{rosenberg1959natural} is the definition of the backbone curve as the frequency-amplitude relationship of a periodic solution family of the conservative unforced limit of the system. Additional arguments are necessary to justify the relevance of these curves for forced-damped vibrations. Hill~et~al.~\cite{Hill_interp_via_Bbk,hill2017identifying}, Kerschen~et~al.~\cite{Kerschen_BasicNNM} and Peeters~et~al.~\cite{peeters_backbone} observe that along each nonlinear normal mode (i.e. periodic orbit) of the conservative limit, weak viscous damping can be cancelled by appropriately chosen external periodic forcing.  Under such forcing, the conservative set of nonlinear normal modes will form the backbone curves. For a general damped and forced nonlinear system however, the relevance of periodic orbits of the conservative limit for the forced response is not well understood. Recently, Hill~et~al.~\cite{hill2017identifying} observed numerically that major parts of such nonlinear normal modes are non-robust and therefore irrelevant for the forced response. They propose a robustness measure to assess the relevance of the conservative nonlinear normal modes for the forced response. Kerschen~et~al.~\cite{Kerschen_BasicNNM} and Peeters~et~al.~\cite{peeters_backbone} mention specific examples in which the forced response of an almost conservative system will be close to the periodic orbits of the conservative system. Since the backbone curve is obtained for the unforced conservative limit in these examples, another method is needed to actually calculate the maximum amplitude for a given forcing. Hill~et~al.~\cite{Hill_interp_via_Bbk} present an energy-transfer-based method for this purpose. They also give, however, a counterexample in which the conservative backbone curve has no relevance for the forced response. 

Parallel to theoretical considerations, backbone curves have been approximated in experiments through the \textit{force appropriation method}. In this method, the  nonlinear system is forced with a harmonic forcing such that the response has a 90-degree phase lag in a modal degree of freedom.  While this force appropriation procedure is plausible for linear viscous damping (or nonlinear damping that is an odd function of the velocities), the approach has remained unjustified for general, nonlinear damping (cf. Peeters~et~al.~\cite{peeters_backbone}). 

An experimental alternative to the force appropriation is the \textit{resonance decay} method, in which the system is forced, such that its response is close to an envisioned invariant surface of the conservative limit. Then the forcing is turned off and the instantaneous amplitude-frequency relationship is identified by signal processing.  Peeters et al.~\cite{peeters_backbone}, however, relate this curve, which is essentially a feature of the damped system, to the orbits of the conservative system  only phenomenologically. 

We also note that force appropriation and the resonance decay aim to reconstruct nonlinear normal modes of the conservative limit. The set formed by these orbits is expected to deviate from the forced response of the actual dissipative system for lager amplitudes and larger damping. As a recent development,  Szalai~et~al.~\cite{article:SSM_SysID} compute the backbone curves from the frequency-amplitude relationship of decaying vibrations on SSMs reconstructed from measured data.  A connection with the backbone curve obtained from the forced response, however, is not immediate.   

In summary, available approaches to compute forced response via model reduction for nonlinear mechanical systems suffer either from heuristic steps or omissions in the reduction procedure, or from a unclear relationship between backbone-curve definitions different from the one relevant for forced-damped vibrations in a practical setting. In the present work, we seek to eliminate these shortcomings simultaneously. First, we employ a mathematically justified reduction process to time-dependent SSMs in the presence of general damping and forcing. Second, with universal, system-independent formulas for the dynamics on the SSM at hand, we derive explicit, leading-order approximations to the actually observed backbone curve of the time-dependent, dissipative response. We show how all this can be achieved  without the use of extensive numerics (such as numerical continuation or numerical time integration) or extensive numerical experimentation (force appropriation and resonance decay). 

Our results are based on a parameterization of an autonomous SSMs that can be continued under the addition of small external forcing~(section 3). Via a simplification of the resulting reduced dynamics on the non-autonomous SSM, we can directly solve for the amplitudes of the forced response, restricting our focus to oscillations near the origin. Without any further restrictions, we calculate backbone curves, stability of the forced response and the amplitude-frequency relationship explicitly~(section 4). We then demonstrate the performance of our explicit backbone-curve formulas in three numerical examples, on which we also compare our results to those obtained from prior methods for approximating forced responses and backbone curves~(section 5).

	\section{Set-up}

We consider a general, quasi-periodically forced, nonlinear, $N$-degree-of-freedom mechanical system of the form	

\begin{equation}
\begin{split}
\mathbf{M}\ddot{\mathbf{q}}+
(\mathbf{C}+\mathbf{G})\dot{\mathbf{q}}+
(\mathbf{K}+\mathbf{N})\mathbf{q}+
\mathbf{f}_{nlin}(\mathbf{q},\dot{\mathbf{q}})
=\varepsilon \mathbf{f}_{ext}(\Omega_1 t,...,\Omega_k t),
\quad \mathbf{q} \in \mathbb{R}^N,
\quad 0 \leq \varepsilon \ll 1,
\label{eq:sys_phys}
\\
\mathbf{f}_{nlin}(\mathbf{q},\dot{\mathbf{q}})=
\mathcal{O}(|\mathbf{q}|^2,|\mathbf{q}||\dot{\mathbf{q}}|,|\dot{\mathbf{q}}|^2),
\qquad
\mathbf{f}_{ext}(\Omega_1 t,...,\Omega_k t)=
\sum_{\mathbf{k}\in \mathbb{Z}^k}\mathbf{f}_{ext}^{\mathbf{k}}e^{i\langle\mathbf{k},\Omega  \rangle t},
\quad k \geq 1,
\end{split}
\end{equation}

\noindent where~$\mathbf{M}$ is a symmetric, positive definite matrix; the stiffness  matrix~$\mathbf{K}$ and the damping matrix~$\mathbf{C}$ are symmetric, positive semi-definite; the matrix of the follower forces~$\mathbf{N}$ and the gyroscopic matrix~$\mathbf{G}$ are skew-symmetric; and the nonlinear forcing vector $\mathbf{f}_{nlin}(\mathbf{q},\dot{\mathbf{q}})$ is at least quadratic in its arguments. Observe, that $\mathbf{q}\!\equiv\!0$ is an equilibrium of the unforced system ($\varepsilon\!=\!0$). The external forcing $\varepsilon\mathbf{f}_{ext}$ does not depend on the generalized coordinates or velocities and has finitely many rationally incommensurate frequencies~($\Omega_1,...,\Omega_k$). As such, $\mathbf{f}_{ext}$ admits a convergent Fourier representation with frequency base vector $\Omega\!=\!(\Omega_1,...,\Omega_k)$, as indicated. 

We denote the eigenvalues of the linearized system~\eqref{eq:sys_phys} by $\lambda_1,...,\lambda_{2N}$, with multiplicities and conjugates included. We assume an underdamped configuration, i.e. complex eigenvalues with nonzero imaginary part and negative real part. Due to the importance of the eigenvalues with the smallest real part for the existence of the non-autonomous SSM (cf. Haller and Ponsioen~\cite{article:Haller_SSMbasics}), we denote one of these eigenvalues by $\lambda_{min}$ and order all eigenvalues as follows:

\begin{equation}
\lambda_j=\overline{\lambda}_{j+N} \qquad 
\mbox{Im}(\lambda_j) > 0,
\qquad
\mbox{Re}(\lambda_{min}) \leq \mbox{Re}(\lambda_j)<0,\qquad j=1,...,N.
\label{eq:cond_eigval}	
\end{equation}

\noindent  By~\eqref{eq:cond_eigval} the $\mathbf{q}\!\equiv\!0$ equilibrium  of the unforced limit of~\eqref{eq:sys_phys} is asymptotically stable. This context is relevant for vibrations of lightly damped structures.

To obtain the first order equivalent system, we define the matrices 
\begin{equation}
\begin{split}
\mathbf{A}
&
=
\begin{pmatrix}
0& \mathbf{I}\\
-\mathbf{M}^{-1}(\mathbf{K}+\mathbf{N})&-\mathbf{M}^{-1}(\mathbf{C}+\mathbf{G})
\end{pmatrix},
\qquad 
\mathbf{G}_{nlin}(\mathbf{x})
=
\begin{pmatrix}
	0\\
	\mathbf{M}^{-1}\mathbf{f}_{nlin}(\mathbf{x})
\end{pmatrix},
\\
\mathbf{g}_{ext}^{\mathbf{k}}
&=
\begin{pmatrix}
0\\
\mathbf{M}^{-1}\mathbf{f}_{ext}^{\mathbf{k}}
\end{pmatrix}
,
\qquad 
\mathbf{G}_{ext}(\Omega_1 t,...,\Omega_k t)=\sum_{\mathbf{k}\in \mathbb{Z}^k}\mathbf{g}_{ext}^{\mathbf{k}}e^{i\langle\mathbf{k},\Omega  \rangle t}.
\end{split}
\label{eq:ss_mat}
\end{equation}
By letting $\mathbf{x}\!=\!(\mathbf{q},\dot{\mathbf{q}})$ in \eqref{eq:sys_phys} and the definitions~\eqref{eq:ss_mat}, we obtain the first-order equivalent system 
\begin{equation}
\dot{\mathbf{x}}
=	
\mathbf{A}\mathbf{x}
+
\mathbf{G}_{nlin}(\mathbf{x})
+
\varepsilon \mathbf{G}_{ext}(\Omega_1 t,...,\Omega_k t).
\label{eq:system_ss}
\end{equation}
\noindent We define the matrices 
\begin{equation}
\mathbf{\Lambda}=\mbox{diag}(\lambda_1,...,\lambda_{2N}), \qquad \mathbf{V}=\left[\mathbf{v}_1,...,\mathbf{v}_{2N}\right],
\qquad
\mathbf{v}_j=\left(
\begin{array}{c}
\mathbf{e}_j\\
\lambda_j \mathbf{e}_j
\end{array}
\right),
\label{eq:mode_shp}
\end{equation}
\noindent where $\mathbf{v}_j$ is the eigenvector of~\eqref{eq:system_ss}, corresponding to the eigenvalue $\lambda_j$ and to the mode shape $\mathbf{e}_j$ of the linear part of~\eqref{eq:sys_phys}.  We assume that the matrix $\mathbf{A}$ is semisimple and therefore $\mathbf{\Lambda}\!=\!\mathbf{V}^{-1}\mathbf{A}\mathbf{V}$ holds. An equivalent autonomous version of the non-autonomous system~\eqref{eq:system_ss} can be obtained by introducing the phases

\begin{equation}
\phi_j=\Omega_j t, \qquad  j=1,...,k,
\label{eq:angles}
\end{equation}

\noindent which yield

\begin{equation}
\begin{pmatrix}
\dot{\mathbf{x}}\\
\dot{\mathbf{\phi}}
\end{pmatrix}
=
\begin{bmatrix}
\mathbf{A}\mathbf{x}
+
\mathbf{G}_{nlin}(\mathbf{x})
+
\varepsilon \mathbf{G}_{ext}(\phi)
\\
\Omega
\end{bmatrix}.
\label{eq:sys_auto}
\end{equation}

For system~\eqref{eq:system_ss} or its equivalent autonomous form~\eqref{eq:sys_auto}, we now restate main results from  of Haller and Ponsioen~\cite{article:Haller_SSMbasics}. We consider eigenspaces of system~\eqref{eq:system_ss} of the form
	\begin{equation}
	E=\mbox{span}\{\mathbf{v}_{1}, ..., \mathbf{v}_{s}, 
	\mathbf{v}_{N+1}, ... \mathbf{v}_{N+s} \},
	\label{eq:modal_subspace}
	\end{equation}
	with their smoothest nonlinear continuation defined as follows. 
	\begin{definition}
	 A \textit{spectral submanifold} (SSM), $W(E)$, corresponding to the eigenspace $E$ defined in~\eqref{eq:modal_subspace} is an invariant manifold of the system~\eqref{eq:system_ss} with the following properties:
	 \begin{enumerate}[(i)]
	 	\item $W(E)$ has the same dimensions as $E$ (i.e. $\mathrm{dim}(W(E))=2s$) and perturbs smoothly from $E$ at $\textbf{x}\!=\!0$ under the addition of the nonlinear and $\mathcal{O}(\epsilon)$ terms of system~\eqref{eq:system_ss};
	 	\item $W(E)$ is strictly smoother than any other invariant manifold satisfying (i).
	 \end{enumerate}
 	\end{definition}
 
From now on, we assume that the non-resonance conditions
\begin{equation}
\sum_{j=1}^s m_j\mbox{Re}(\lambda_j)  \neq  \mbox{Re}(\lambda_n),
\qquad n = s+1,...,N,
\qquad  2 \leq \sum_{j=1}^s m_j \leq \Sigma(E),
\qquad m_j \in \mathbb{N},
\label{eq:nonres_cond}
\end{equation}
 hold,  with the absolute spectral quotient $\Sigma(E)$ defined as
\begin{equation}
\Sigma(E)=\mbox{Int}
\left( \frac{  \mbox{Re}(\lambda_{min})}{\max\limits_{j=1,...,s}(\mbox{Re}(\lambda_j))}
\right),  
\end{equation}
where the operator $\mbox{Int}(\cdot)$ extracts the integer part of its argument. Then we have the following results on the SSMs of the general mechanical system~\eqref{eq:system_ss};
\begin{theorem}
Assume that the non-resonance conditions~\eqref{eq:nonres_cond} are satisfied for an eigenspace $E$ defined in~\eqref{eq:modal_subspace}. Then the following statements hold: 
 \begin{enumerate}[(i)]
	\item The SSM, $W(E)$, for system~\eqref{eq:system_ss} uniquely exists in the class of $C^{\Sigma(E)+1}$ manifolds.
	\item A parameterization $\mathbf{W}: ~\mathbb{R}^{2s}\rightarrow \mathbb{R}^{2N}$ of the invariant manifold $W(E)$ can be approximated in a neighborhood of the origin as a polynomial in the parameterization variable $\mathbf{z}$, with coefficients depending on the phase variables $\mathbf{\phi}$, i.e.,  
	\begin{equation}
	\mathbf{x}=\mathbf{W}(\mathbf{z},\mathbf{\phi}), \qquad \mathbf{z}\in \mathbb{R}^{2s},
	\end{equation} 
	\item There exist a polynomial function  $\mathbf{R}(\mathbf{z},\mathbf{\phi})$, defined on an open neighborhood of $\mathbf{x}\!=\!0$, such that the invariance condition
	\begin{equation}
	\mathbf{A}\mathbf{W}(\mathbf{z},\mathbf{\phi})+
	\mathbf{G}(\mathbf{W}(\mathbf{z},\mathbf{\phi}))+
	\varepsilon\mathbf{G}_{ext}(\mathbf{\phi})
	=
	D_{\mathbf{z}}\mathbf{W}(\mathbf{z},\mathbf{\phi})\mathbf{R}(\mathbf{z},\mathbf{\phi})
	+D_{\mathbf{\phi}}\mathbf{W}(\mathbf{z},\mathbf{\phi})\Omega.
	\label{eq:inv_cond_SSM}
	\end{equation}
	holds. Therefore, the dynamics on the SSM (i.e., the reduced dynamics) are governed by
	\begin{equation}
	\dot{\mathbf{z}}=\mathbf{R}(\mathbf{z},\mathbf{\phi}).
	\end{equation} 
	\item The parameterization $\mathbf{W}(\mathbf{z},\phi)$, as well the reduced dynamics $\mathbf{R}(\mathbf{z},\phi)$, are robust with respect to changes in the parameters. 
\end{enumerate}
\label{Thm:SSM}
\end{theorem}
	
  \begin{proof} This is a restatement of the main theorem by Haller and Ponsioen~\cite{article:Haller_SSMbasics} (Theorem 4), deduced from more abstract results on invariant manifolds of Haro and de la Lave~\cite{HARO_SSM} (Theorem 4.1) in our current setting.
 \end{proof}

 If the non-resonance conditions are satisfied for the general mechanical system~\eqref{eq:sys_phys}, Theorem~\ref{Thm:SSM} establishes the existence, smoothness and uniqueness of the SSM tangent to a modal subspace of interest. Due to the smooth persistence of the SSM with respect to the small parameter $\varepsilon$, the parameterization of the SSM, as well the reduced dynamics, can be expanded in $\varepsilon$. Since the forcing in eq.~\eqref{eq:system_ss} is of the first order in $\varepsilon$, the leading-order approximations to the spectral submanifold ($\mathbf{W}_0$) and to the dynamics~($\mathbf{R}_0$) do not depend on the phase variables $\mathbf{\phi}$. Specifically, we have
\begin{equation}
\begin{split}
\mathbf{W}(\mathbf{z},\mathbf{\phi})&=
\mathbf{W}_0(\mathbf{z})+\sum_{l=1}^{\infty} \varepsilon^l \mathbf{W}_l(\mathbf{z},\mathbf{\phi})
,
\\
\mathbf{R}(\mathbf{z},\mathbf{\phi})&=
\mathbf{R}_0(\mathbf{z})+\sum_{l=1}^{\infty}\varepsilon^l \mathbf{R}_l(\mathbf{z},\mathbf{\phi})
,
\end{split}
\label{eq:expans_SSM+dyn}
\end{equation}
where the subscripts of $\mathbf{W}$ and $\mathbf{R}$ indicate the order in the $\varepsilon$ expansion. As a consequence, we have the following corollary.
\begin{corollary}
 If in addition to the inner non-resonance conditions~\eqref{eq:nonres_cond}, the outer  non-resonance conditions
\begin{equation}
\sum_{j=1}^s m_j \lambda_j \neq \lambda_n, 
\qquad n =1,...,s,
\qquad  2 \leq \sum_{j=1}^s m_j \leq \Sigma(E),
\qquad m_j \in \mathbb{N},
\label{eq:simplify_cond}
\end{equation}
\noindent hold for the eigenspace $E$ defined in~\eqref{eq:modal_subspace}, then $\mathbf{R}_0(\mathbf{z},\phi)$ can be chosen linear in $\mathbf{z}$. 
\end{corollary}

\begin{proof}
This corollary follows directly from the work of Cabre et al.~\cite{Cabre_para} for discrete mappings (Theorem 1.1) and is also stated by Szalai~et~al.~\cite{article:SSM_SysID}. These results are applicable here since $\mathbf{W}_0$ and $\mathbf{R}_0$ are autonomous. 
\end{proof}

To conveniently express the polynomial expansions of $\mathbf{W}(\mathbf{z},\mathbf{\phi})$ and $\mathbf{R}(\mathbf{z},\mathbf{\phi})$,  we use the multi-index notation 
\begin{equation}
\begin{split}
\mathbf{W}_l(\mathbf{z},\mathbf{\phi})
&
=\sum_{\mathbf{m}\in \mathbb{N}^{2s}_0} \mathbf{w}_l^{\mathbf{m}} (\mathbf{ \phi}) \mathbf{z}^{\mathbf{m}},
\qquad \mathbf{w}_l^{\mathbf{m}} \in \mathbb{C}^{2N},
\\
\mathbf{R}_l(\mathbf{z},\mathbf{\phi})
&=\sum_{\mathbf{m}\in \mathbb{N}^{2s}_0} \mathbf{r}_l^{\mathbf{m}} (\mathbf{ \phi}) \mathbf{z}^{\mathbf{m}},
\qquad \mathbf{r}_l^{\mathbf{m}} \in \mathbb{C}^{2s},
\end{split}
\label{eq:multiindex_exp}
\end{equation}
 where the superscript $\mathbf{m}$ indicates the associated monomial of the coefficient vectors $\mathbf{w}_l^{\mathbf{m}}$ and $\mathbf{r}_l^{\mathbf{m}}$.

 \section{Spectral submanifolds for the forced system}

Given a parameterization of the SSM, $W(E)$, and its reduced dynamics for the autonomous limit of~\eqref{eq:sys_phys} ($\varepsilon\!=\!0$), we now consider the continuation of these under the addition of small forcing terms. We truncate the parameterization $\mathbf{W}(\mathbf{z},\mathbf{ \phi})$ and the associated reduced dynamics $\mathbf{R}(\mathbf{z},\mathbf{ \phi})$ at $\mathcal{O}(\varepsilon|\mathbf{z}|,\varepsilon^2)$. With the notation~\eqref{eq:multiindex_exp}, the series expansion~\eqref{eq:expans_SSM+dyn} of $\mathbf{W}(\mathbf{z},\mathbf{\phi})$ and~$\mathbf{R}(\mathbf{z},\mathbf{\phi})$ can be rewritten as 
	\begin{equation}
	\begin{split}
	\mathbf{W}(\mathbf{z},\mathbf{\phi})&=
	\mathbf{W}_0(\mathbf{z})+\varepsilon\mathbf{w}^{\mathbf{0}}_1(\mathbf{\phi}) 
	+\mathcal{O}(\varepsilon|\mathbf{z}|,\varepsilon^2),
	\\
	\mathbf{R}(\mathbf{z},\mathbf{\phi})&=
	\mathbf{R}_0(\mathbf{z})+\varepsilon\mathbf{r}^{\mathbf{0}}_1(\mathbf{\phi}) 
	+\mathcal{O}(\varepsilon|\mathbf{z}|,\varepsilon^2). 
	\end{split}
	\label{eq:expans_SSM+dyn_detail_epsz}
	\end{equation} 
The equations~\eqref{eq:expans_SSM+dyn_detail_epsz} reveal that only the unknown coefficient vectors~$\mathbf{w}^{\mathbf{0}}_1(\mathbf{\phi})$ and~$\mathbf{r}^{\mathbf{0}}_1(\mathbf{\phi})$  need to be computed to achieve the desired $\mathcal{O}(\varepsilon|\mathbf{z}|,\varepsilon^2)$ accuracy.

First, we discuss a general leading-order parameterization $\mathbf{W}_0$ and its dynamics $\mathbf{R}_0$, then we modify this parameterization to accomodate the near-resonant nature of conjugate eigenvalue pairs that arises under weak damping (cf.~Szalai~et~al.~\cite{article:SSM_SysID}). 

\subsection{General parameterization}

 For a general parameterization truncated at $\mathcal{O}(\varepsilon|\mathbf{z}|,\varepsilon^2)$, we state the result in the following lemma;
 
 \begin{lemma}
 If the non-resonance conditions~\eqref{eq:nonres_cond} are satisfied for the spectral subspace~\eqref{eq:modal_subspace} of system~\eqref{eq:system_ss}, then the coefficient vectors~$\mathbf{w}^{\mathbf{0}}_1(\mathbf{\phi})$ and~$\mathbf{r}^{\mathbf{0}}_1(\mathbf{\phi})$ of  parameterization $\mathbf{W}(\mathbf{z},\mathbf{\phi})$  and the dynamics $\mathbf{R}(\mathbf{z},\mathbf{\phi})$ are given by 
 \label{lem:gen_para}
\begin{subequations}
\begin{align}
\mathbf{w}^0_1&=
\sum_{\mathbf{k}\in\mathbb{Z}^k} 
\mathbf{V}
(i \langle \mathbf{k},\Omega\rangle \mathbf{I} -\mathbf{\Lambda})^{-1}
\mathbf{V}^{-1}
\mathbf{g}_{ext}^{\mathbf{k}}
e^{i\langle\mathbf{k},\phi\rangle}, 
\label{eq:w1_harm}
\\
\mathbf{r}^0_1&=\mathbf{0}.
\label{eq:r1_harm}
\end{align}
\end{subequations}
 \end{lemma} 

\begin{proof}
	The non-resonance conditions~\eqref{eq:nonres_cond} ensure the existence of the SSM, therefore the parameterization and the reduced dynamics can be expressed in the form~\eqref{eq:expans_SSM+dyn}. Substituting this series expansion into the invariance condition~\eqref{eq:inv_cond_SSM} and comparing terms of equal order in $\varepsilon$ and $\mathbf{z}$, we obtain the expressions~\eqref{eq:w1_harm} and~\eqref{eq:r1_harm}. We detail this coefficient comparison in Appendix~\ref{app:der_res1}.
\end{proof}

The specific form of $\mathbf{W}_0$ and $\mathbf{R}_0$ depends on the choice of the modal subspace~\eqref{eq:modal_subspace}. Cabre~et~al.~\cite{Cabre_para} point out that the parameterization of SSM is not unique, even though the SSM is. Because of the conditions \eqref{eq:cond_eigval}, the inverse in formula \eqref{eq:w1_harm} is nonsingular. Still, if the norm of the damping matrix $\mathbf{C}$ is small and a harmonic $\langle \mathbf{k},\Omega\rangle $ is near-resonant, i.e., 

\begin{equation}
\langle \mathbf{k},\Omega \rangle \approx \mbox{Im}(\lambda_l),
\label{eq:neqr_res_forcing}
\end{equation}

\noindent then small denominators arise in eq.~\eqref{eq:w1_harm}. These denominators would restrict the domain of validity of our calculations. To avoid this issue, we will eliminate small denominators by keeping terms in  $\mathbf{R}(\mathbf{z},\mathbf{\phi})$, that could otherwise be eliminated from the reduced dynamics.

\subsection{Forced response of the nonlinear mechanical system}

Having identified terms that potentially contain small denominators, we continue by keeping additional terms in the reduced dynamics to ascertain that no small denominators arise in the parameterization. In order to construct frequency-amplitude response curves, we now assume canonical single-harmonic forcing ($k\!=\!1$)  in the form of

\begin{equation}
\mathbf{f}_{ext}=\mathbf{f}\cos(\Omega t)=\mathbf{f}\frac{e^{i\Omega t}+e^{-i\Omega t}}{2}.
\label{eq:forcing}
\end{equation}  

\noindent Therefore, only the forcing terms $\mathbf{g}_{ext}^{\pm1}$ (cf. eq.~\eqref{eq:ss_mat}) are nonzero. The period of the forcing~\eqref{eq:forcing} is $T\!=\!2\pi/\Omega$.  We restrict our calculations to the case when $W(E)$ is two-dimensional ($s\!=\!1$), which are tangent to an eigenspace

\begin{equation}
E_l=\mbox{span}\{\mathbf{v}_{l}, \mathbf{v}_{l+N}\}.
\label{eq:2D_eigspace}
\end{equation}

We denote the parameterization variable for the corresponding SSM $W(E_l)$ by $\mathbf{z}\!=\!\left[z_l, \overline{z}_l\right]^T$. Since the eigenvalues $\lambda_l$ and $\lambda_{l\!+\!N}$  are complex conjugates~(cf.~condition~\eqref{eq:cond_eigval}), the internal resonance conditions~\eqref{eq:simplify_cond} are technically satisfied and the dynamics $\mathbf{R}_0(\mathbf{z})$ could be chosen linear. As noted by Szalai~et~al.~\cite{article:SSM_SysID}, however, the near-resonance relationships

\begin{equation}
(m+1)\lambda_l+m\lambda_{l+N}\approx\lambda_l,
\qquad
m\lambda_l+(m+1)\lambda_{l+N}\approx\lambda_{l+N},
\qquad m \leq M
\end{equation} 

\noindent between complex conjugate eigenvalues always hold for small damping (i.e. $2M|\mbox{Re}(\lambda_l)\!| \ll\! 1 $). The weaker the damping, the higher the value of the positive integer $M$ needs to be set. Removing the corresponding terms from the dynamics would lead to small denominators in the parameterization $\mathbf{W}_0$ of the SSM. To this end, we keep such near-resonant terms in $\mathbf{R}_0(\mathbf{z})$ by letting

\begin{equation}
\mathbf{R}_0(\mathbf{z})= 
\begin{bmatrix}
\lambda_l z_l\\
\bar{\lambda}_{l} \bar{z}_l
\end{bmatrix}
+\sum_{m=1}^M
\begin{bmatrix}
\beta_m z^{m+1}_l\bar{z}^{m}_l\\
\overline{\beta}_m z^{m}_l\bar{z}^{m+1}_l
\end{bmatrix}+\mathcal{O}(\mathbf{z}^{2M+3}).
\label{eq:R0}
\end{equation}

The order of the autonomous SSM and its associated dynamics in the parameterization variable $\mathbf{z}$ is $2M\!+\!1$. For instance, for the choice of  $M\!=\!1$, a parameterization the autonomous SSM  $\mathbf{W}_0(\mathbf{z})$  and its associated dynamics $\mathbf{R}_0(\mathbf{z})$ are of order three in $\mathbf{z}$. For this case, a formula for the constant $\beta_1$ in~\eqref{eq:R0} for the general mechanical system~\eqref{eq:system_ss} with diagonalized linear part is given by Szalai~et~al.~\cite{article:SSM_SysID}, which we recall in Appendix~\ref{app:Coeff_Szalai} for completeness. 

For increasing accuracy or large amplitude oscillations it is desirable to compute~\eqref{eq:R0} for a higher choice of $M$ ($M\!>\!1$). To compute the arising constants $\beta_m$ of the reduced dynamics~\eqref{eq:R0} the invariance condition (cf. Appendix~\ref{app:der_res1} eq.~\eqref{eq:inv_first_order}) has to be solved for a polynomial  $\mathbf{W}_0(\mathbf{z})$ and $\mathbf{R}_0(\mathbf{z})$ manually or the automated computation package of Ponsioen et al.~\cite{Sten_auto_SSM} can be utilized. For the calculation of the $\mathcal{O}(5)$ SSM  ($M\!=\!2$), we provide a \textsc{Matlab} script as electronic supplementary material.  

As for the computation of $\mathbf{W}_0$ in the $M\!=\!1$ case, Szalai~et~al.~\cite{article:SSM_SysID}, showed that a  two-dimensional SSM,  $W(E_l)$, of the unforced limit of system~\eqref{eq:sys_phys} can be constructed, if the further non-resonance conditions

\begin{equation}
m_1\lambda_l+m_2\bar{\lambda}_l\not\approx \lambda_j, \qquad j\neq l,l+N,\qquad 1 \leq m_1+m_2 \leq \Sigma(E_l),
\label{eq:ext_nearres_cond}
\end{equation}

\noindent are satisfied.

To study the continuation of the autonomous SSM from Szalai~et~al.~\cite{article:SSM_SysID} under the addition of the small forcing terms defined in~\eqref{eq:forcing}, we rescale the parameterization variable 

\begin{equation}
\mathbf{z}\mapsto \varepsilon^{\frac{1}{2M+2}}\mathbf{z}=\mu \mathbf{z},
\label{eq:rescaling}
\end{equation}

\noindent  and truncate all formulas for the SSM and its reduced dynamics at order $\mu^{2M+3}$ in the following. Higher-order approximations could be obtained in a similar fashion.

To  explicitly construct an approximation to the SSM, we define the matrices $\mathbf{S}^{+}$ and $\mathbf{S}^{-}$ elementwise as
 	\begin{equation}
  \mathbf{S}^{+}_{jm}=\delta_{jm}-\delta_{lj}\delta_{lm}, 
 	\qquad
 	\mathbf{S}^{-}=\delta_{jm}-\delta_{(l+N)j}\delta_{(l+N)m},
 	\qquad j,m=1,...,2N.
 	\label{eq:aux_matrix}
 \end{equation}	
 Both matrices ($\mathbf{S}^{+}$ and $\mathbf{S}^{-}$) equal the identity, except that the element $S^+_{ll}$  is zero and the $(l\!+\!N)$-th entry on the main diagonal of $\mathbf{S}^{-}$ is zero. Furthermore, we denote the $j$-th row of the inverse of the eigenvector matrix $\mathbf{V}$ by $\mathbf{t}_j$, i.e.,
 \begin{equation}
 [\mathbf{t}_1,\mathbf{t}_2,...,\mathbf{t}_{2N}]=\mathbf{V}^{-1},\qquad \mathbf{t}_j\in\mathbb{C}^{1\times 2N},\qquad j=1,...,2N.
 \label{eq:inverse_mat}
 \end{equation}
 	We then have the following result for the autonomous SSM and its associated reduced dynamics.

 \begin{theorem}
 If the non-resonance conditions \eqref{eq:nonres_cond} and \eqref{eq:ext_nearres_cond} hold for the subspace $E_L$	(cf. eq.~\eqref{eq:2D_eigspace}) for the general mechanical system~\eqref{eq:sys_phys} under the canonical single-harmonic forcing~\eqref{eq:forcing}, then the $\mathcal{O}(\mu^{2M+3})$ approximation of the parameterization and its reduced dynamics can be written in the form
\begin{subequations}
	\begin{align}
	\mathbf{W}(\mathbf{z},\phi)&=\mathbf{W}_0(\mathbf{z})+\mu^{2M+2}\mathbf{w}^{\mathbf{0}}_1(\phi)+\mathcal{O}(\mu^{2M+3}),
	\label{eq:SSM_nearres}
	\\
	\mathbf{R}(\mathbf{z},\Omega t)&=
	\mathbf{R}_0(\mathbf{z})+\mu^{2M+2}\mathbf{r}^{\mathbf{0}}_1(\Omega t)+\mathcal{O}(\mu^{2M+3}),
	\label{eq:dynamics_R_nearres}
	\end{align}
\end{subequations}
where the coefficient vectors $\mathbf{w}^{\mathbf{0}}_1$ and $\mathbf{r}^{\mathbf{0}}_1$ are given by 
\begin{subequations}
	\begin{align}
		\mathbf{w}^{\mathbf{0}}_1&=  
\mathbf{V S}^{+}
(i \Omega\mathbf{I}-\mathbf{\Lambda})^{-1}\mathbf{V}^{-1}
\mathbf{g}_{ext}^{(1)}
e^{i\Omega t}
+  
\mathbf{V S}^{-}
(-i \Omega\mathbf{I}-\mathbf{\Lambda})^{-1}\mathbf{V}^{-1}
\mathbf{g}_{ext}^{(-1)}
e^{- i\Omega t},
\label{eq:SSM_nearres_sol}
\\
\mathbf{r}^{\mathbf{0}}_1&=
r_c 
\begin{bmatrix}
e^{i\Omega t}\\
-e^{-i\Omega t}
\end{bmatrix},
\qquad 
r_{c}= \mathbf{t}_l
\mathbf{g}_{ext}^{(1)}.
\label{eq:Red_dyn_forcing}
\end{align}
\end{subequations}
\label{Thm:auto_SSM}
\end{theorem}
		
\begin{proof}
Since the non-resonance conditions~\eqref{eq:nonres_cond} hold the existence of the SSM can be guaranteed. Given that the additional non-resonance conditions~\eqref{eq:ext_nearres_cond} also hold, the result from Szalai~et~al.~\cite{article:SSM_SysID} also applies and we can select the parameterization variable such that the reduced dynamics of the autonomous limit of system~\eqref{eq:sys_phys} is of the form~\eqref{eq:R0}. Substituting the series expansion~\eqref{eq:expans_SSM+dyn} and the scaling~\eqref{eq:rescaling} into the invariance condition~\eqref{eq:inv_cond_SSM} and comparing terms of equal order in $\mu$, we obtain eqs.~\eqref{eq:SSM_nearres} and~\eqref{eq:dynamics_R_nearres}. We solve the arising equation at order $\mu^{2M+2}$ eliminating small denominators and obtain the explicit equations for $\mathbf{w}^{\mathbf{0}}_1$ and $\mathbf{r}^{\mathbf{0}}_1$ (cf. eqs.~\eqref{eq:SSM_nearres_sol} and~\eqref{eq:Red_dyn_forcing}). We give the detailed derivations in Appendix~\ref{app:der_res2}. 
\end{proof}

\begin{remark}
The constant $r_c$ in eq.~\eqref{eq:Red_dyn_forcing} is the component of the forcing vector $\mathbf{f}$ in eq.~\eqref{eq:forcing} falling in the subspace $E_l$ defined in~\eqref{eq:2D_eigspace}.
\end{remark}

\begin{remark}
Since we assume nonzero real part for the eigenvalues $\lambda_j$ (cf. condition~\eqref{eq:cond_eigval}), the inverse in~\eqref{eq:SSM_nearres_sol} is nonsingular. By construction, the matrices $\mathbf{S}^{\pm}$ cancel out the terms with small denominators in the parameterization.
\end{remark}

\begin{remark}
The non-resonance conditions~\eqref{eq:ext_nearres_cond} are violated for internally resonant structures. In this case, the system dynamics cannot be reduced to a two-dimensional SSM; rather, a higher-dimensional SSM needs to be constructed. Specific formulas for the reduction of an autonomous system to a higher-dimensional SSM ($\mathbf{W}_0$ and $\mathbf{R}_0$, for  $s\!>\!1$) have not yet been obtained in the literature, even though they can, in principle,  be deduced from the invariance condition~\eqref{eq:inv_cond_SSM}. 
\end{remark}

Theorem~\ref{Thm:auto_SSM} leads to have the following corollary.

\begin{corollary}
The eigenvectors~$\mathbf{v}_j$ can be normalized such that $r_c$ is purely imaginary. 
\label{cor:eig_vecs}
\end{corollary}

\begin{proof}
	The proof relies on the fact that we can multiply the eigenvectors with a complex constant such that eq.~\eqref{eq:decomplex_r} holds. We detail this in Appendix~\ref{app:der_res3}. 
\end{proof} 

\begin{remark}
In the case of purely symmetric system matrices in~\eqref{eq:sys_phys} ($\mathbf{N}\!=\!\mathbf{0}$ and $\mathbf{G}\!=\!\mathbf{0}$) and structural damping ($\mathbf{C}=\alpha_m \mathbf{M}+\alpha_k\mathbf{K}~~\alpha_k,\alpha_m\in\mathbb{R} $), the mode shapes~$\mathbf{e}_j$~(cf. eq.~\eqref{eq:mode_shp}) can be mass-normalized, i.e. for the matrix of mode shapes $\mathbf{E}\!=\!\left[\mathbf{e}_1,...,\mathbf{e}_N\right]$ 
\begin{equation}
\mathbf{E}^{-1}\mathbf{ME}=\mathbf{I},
\label{eq:massnormalized}
\end{equation}
holds. Then the constant $r_c$ turns out to be always purely imaginary, which we also derive in Appendix~\ref{app:der_res3}.
\label{rmk:modeshapes}     
\end{remark}

Based on Corollary~\ref{cor:eig_vecs}, we can assume a purely imaginary constant $r_c$ without loss of generality. We denote the imaginary part of $r_c$ by $r$, i.e.,
\begin{equation}
r=\mbox{Im}(r_c).
\label{eq:decomplex_r}
\end{equation}
To determine the steady state response of~\eqref{eq:sys_phys}, we seek for $T$-periodic orbits of~the reduced dynamics~\eqref{eq:dynamics_R_nearres}. To this end, we transform the parameterization variables to polar coordinates by letting
\begin{equation}
z_l=\rho e^{i\theta}.
\label{eq:pol_trans}
\end{equation}
Furthermore, we separate the real and imaginary parts of the reduced dynamics~\eqref{eq:R0}  as 
\begin{subequations}
	\begin{align}
	\mbox{Re}(\mathbf{R}_0(\mathbf{z}))=a(\rho)&=\mbox{Re}(\lambda_l)\rho+\sum_{m=1}^M \mbox{Re}(\beta_m)\rho^{2m+1},
	\label{eq:rel_dyn}
	\\
	\frac{1}{\rho}\mbox{Im}(\mathbf{R}_0(\mathbf{z}))=b(\rho)&=\mbox{Im}(\lambda_l)+\sum_{m=1}^M \mbox{Im}(\beta_m)\rho^{2m}.
	\label{eq:Backbone}
	\end{align}
\end{subequations}
By formula~\eqref{eq:Red_dyn_forcing}, if the forcing vector $\mathbf{f}$ is perpendicular to the subspace $E_l$, then $r_c$ is zero. In that case, system~\eqref{eq:dynamics_R_nearres} has a fixed point at the origin, which is asymptotically stable, because of the conditions~\eqref{eq:cond_eigval}. In general, however, $r$ is nonzero, in which case we obtain the following:
\begin{theorem}
  With the transformation~\eqref{eq:pol_trans} and the notation introduced in~\eqref{eq:rel_dyn} and~\eqref{eq:Backbone}, the following specific expressions for $T$-periodic orbits of the reduced dynamics~\eqref{eq:dynamics_R_nearres} on the time dependent SSM~$W(E_l)$ for nonzero $r$ hold:
\begin{enumerate}[(i)]
	\item \textit{Amplitude of the periodic response}: The amplitudes of the $T$-periodic orbits of \eqref{eq:dynamics_R_nearres} are given by the zeros of the equation
	\begin{equation}
	f(\rho,\Omega):=\left[a(\rho)\right]^2+\left[b(\rho)-\Omega\right]^2\rho^2-\varepsilon^2r^2.
	\label{eq:ploy2solvem}
	\end{equation}
	\item \textit{Phase shift of the periodic response}: For a given amplitude $\rho$ of the periodic response, the phase shift $\psi$ between the $T$-periodic orbit and the external forcing $\mathbf{f}_{ext}$ is 
	\begin{equation}
	\psi=\arccos \left( \frac{\left[\Omega-b(\rho)\right]\rho}{\varepsilon r}\right).
	\label{eq:phase_ss}
	\end{equation}
	\item \textit{Stability of the periodic response}: The stability of the $T$-periodic response with amplitude $\rho$ is determined by the eigenvalues of the Jacobian
	\begin{eqnarray}
	\mathbf{J}(\rho)&=
	\begin{bmatrix}
	\frac{\partial a(\rho)}{\partial \rho} &
	\left[\Omega-b(\rho)\right]\rho\\
	\frac{\partial b(\rho)}{\partial \rho}  -\frac{\Omega-b(\rho)}{\rho}&
	\frac{a(\rho)}{\rho}
	\end{bmatrix}.
	\label{eq:jacobian}
	\end{eqnarray}
\end{enumerate}
\end{theorem}

\begin{proof}
This result can be deduced by substitution of the transformation~\eqref{eq:pol_trans} into the reduced dynamics~\eqref{eq:dynamics_R_nearres} and solving the resulting equations for $T$-periodic orbits. We carry out these computations in detail in Appendix~\ref{app:der_res3}.
\end{proof}

The constants $\beta_m$, necessary to compute $a(\rho)$ and $b(\rho)$, can be obtained from the invariance of the SSM (cf. eq.~\eqref{eq:inv_cond_SSM}). Depending on the order of the SSM ($M$) specific formulas for $\beta_m$ can be taken from Szalai~et~al.~\cite{article:SSM_SysID} or Appendix~\ref{app:Coeff_Szalai} ($M\!=\!1$), the \textsc{Matlab} script provided as electronic supplementary material ($M\!=\!2$) or the automated computation package of Ponsioen et al.~\cite{Sten_auto_SSM}.

 \section{Analytic results on backbone curve, periodic responses and their stability}

Having derived condensed formulas for the amplitude~\eqref{eq:ploy2solvem} and the stability~\eqref{eq:jacobian} of the forced response of system~\eqref{eq:sys_phys}, we can now analytically compute backbone curves and stability regions. Furthermore, we obtain below the forced response in physical coordinates.  

\subsection{Backbone curve}

As mentioned in the Introduction various definitions of the backbone curve can be found in the literature. The definition by Klotter~\cite{Klotter}, as the frequency-amplitude relationship of the conservative unforced limit, was adopted by Rosenberg and Atkinson~\cite{rosenberg1959natural}. This definition, however, has two major drawbacks. First, a general justification for the relevance of this curve for the response of the forced-damped system~\eqref{eq:sys_phys} is not available to the best of our knowledge. Furthermore, due to the no-damping assumption, it is challenging to observe this curve experimentally. The definition by Nayfeh and Mook~\cite{nayfeh_nonlinear} and Cveticanin~et al.~\cite{cveticanin2017dynamics} of the backbone curve as the curve connecting points of maximal response amplitude as a function of an external forcing frequency, defines a relevant and experimentally observable curve. We formalize this definition here as follows.
\begin{definition}
	The \textit{backbone curve} of the mechanical system~\eqref{eq:sys_phys} is the curve of maximal amplitude of the periodic response on the SSM~\eqref{eq:SSM_nearres} as function of the frequency of the external forcing~\eqref{eq:forcing}.
\label{def:backbone} 
\end{definition}

The function~\eqref{eq:ploy2solvem} relates implicitly the response amplitude $\rho$ with the forcing amplitude $r$ and frequency $\Omega$, and hence summarzises information about a whole family of response curves. The maximal amplitude location of each such curve is a single point on the backbone curve  by Definition~\ref{def:backbone}.  By  Definition~\ref{def:backbone}, points on the backbone curve can be identified by equating the derivative of the  amplitudes with respect to the forcing frequency $\Omega$ with zero. To find these locations, first note that implicit differentiation of~\eqref{eq:ploy2solvem} gives

\begin{equation}
\frac{\partial f(\rho,\Omega)}{\partial \Omega}=-2(b(\rho)-\Omega)\rho^2.
\label{eq:partial_f_Bbk}
\end{equation} 

\noindent To identify the frequency $\Omega\!=\!\Omega_{max}$, at which the amplitude of the forced response of~\eqref{eq:sys_phys} is at a maximum, we equate the expression~\eqref{eq:partial_f_Bbk} with zero. Solving for $\Omega_{max}$ from the resulting equation, we obtain

\begin{equation}
\Omega_{max}(\rho_{max})=b(\rho_{max})=\mbox{Im}(\lambda_l)+\sum_{m=1}^M \mbox{Im}(\beta_m)\rho^{2m}_{max}.
\label{eq:freq_of_backbone}
\end{equation}

The maximal response amplitude $\rho_{max}$ parameterizes the backbone curve. The phase shift between response and excitation along the backbone curve is given by

\begin{equation}
\psi(\rho_{max})=\frac{\pi}{2},
\label{eq:pahse_lag_at_Bbk}
\end{equation} 

\noindent as one obtains from~\eqref{eq:phase_ss} by substituting $\rho\!=\!\rho_{max}$ and $\Omega=\Omega_{max}$. Peeters et~al.~\cite{peeters_backbone} derived a similar 90-degree phase lag, under the assumption of structural damping. Equation~\eqref{eq:pahse_lag_at_Bbk} confirms this conclusion for any damping, that is a polynomial function of positions and velocities.

The SSM construction described by Haller and Ponsioen~\cite{article:Haller_SSMbasics} for dissipative systems does not apply to that limit as the $\mathbf{q}\!=\!0$ equilibrium is not hyperbolic in that case. The Lyapunov subcenter-manifold theorem for autonomous conservative systems (cf. Kelley~\cite{kelley_LSM}), however, guarantees the existence of an unique analytic invariant manifold tangent to the modal subspace~\eqref{eq:2D_eigspace} under appropriate non-resonance conditions. These Lyapunov subcenter-mainfolds (LSMs) are then filled with periodic orbits. If, in addition to the forcing, the linear and nonlinear damping are also of first order in $\varepsilon$ , the $\varepsilon \rightarrow 0$ limit of system~\eqref{eq:sys_phys} is conservative and unforced. Then, by the uniqueness of the LSM (cf. Kelley~\cite{kelley_LSM}) and the continuity of the expansions~\eqref{eq:expans_SSM+dyn} of the SSM, it is reasonable to expect that the SSM limits on the LSM. A mathematical proof for this expectation, however, is not available yet. 

We obtain the conservative limit of the reduced dynamics \eqref{eq:dynamics_R_nearres}  by taking the limits $\varepsilon \!\rightarrow \! 0$, $\mbox{Re}(\lambda_l)\!\rightarrow \! 0$ and $\mbox{Re}(\beta_m)\!\rightarrow \! 0$ ($m=1,...,M$). Transforming this limit to poolar coordinates, we obtain the same frequency-amplitude relationship as given by the backbone curve~\eqref{eq:freq_of_backbone}. Therefore we can confirm analytically that the frequency-amplitude relationship of the conservative limit is an $\mathcal{O}(\mu^{2M+3})$ approximation to actual backbone curve. The closeness of the two curves assumed by, e.g, the resonance decay method, has only been argued heuristically by Peeters et al.~\cite{peeters_backbone}.

We further note that the backbone curve~\eqref{eq:freq_of_backbone} is the same as derived by Szalai et al.~\cite{article:SSM_SysID}, who define the backbone curve as the frequency-amplitude relationship of the decaying response along an SSM. From their calculations however, the relevance of this curve to the forced response of system~\eqref{eq:sys_phys} is not immediate. Our derivations clarify here this relevance.

The backbone curve~\eqref{eq:freq_of_backbone} is independent of the forcing amplitude. This fact is clear for the undamped and unforced frequency-amplitude relationship, as there is no forcing in the system, but the same result also follows directly from our analytical calculations for the damped-forced mechanical system~\eqref{eq:sys_phys}. The forcing amplitude determines the location along the backbone curve, where the maximum of the response curve can be found. To obtain the maximum response amplitude for a given forcing equation \eqref{eq:ploy2solvem} has to be solved. Along the backbone curve~\eqref{eq:ploy2solvem} simplifies to 

\begin{equation}
f(\rho_{max})=\left[\mbox{Re}(\lambda_l)\rho_{max}+\sum_{m=1}^M \mbox{Re}(\beta_m)\rho_{max}^{2m+1}\right]^2-\varepsilon^2r^2,
\label{eq:max_at_backbone}
\end{equation}

\noindent which can have multiple solutions for $\rho_{max}$.

We obtain a parameterization of the forced response in the $(\rho,\Omega)$ parameter space, by solving eq.~\eqref{eq:ploy2solvem} for $\Omega$   

\begin{equation}
\Omega(\rho)=b(\rho)\pm \frac{1}{\rho} \sqrt{\varepsilon^2 r^2-a(\rho)^2},
\label{eq:Omega_of_rho}
\end{equation}

\noindent where only real values of $\Omega$ are meaningful. Equation \eqref{eq:Omega_of_rho} reveals that the forced response is symmetric with respect to the backbone curve (cf. Fig.~\ref{fig:sketch_of_f_res}). For a given amplitude $\rho\!=\!\rho_0$  one or two forced responses with that amplitude may exist.  If there is only one such response, it must lie on the backbone curve~\eqref{eq:freq_of_backbone}. 

In practice, the forcing frequency $\Omega$ is known and the amplitude $\rho$ needs to be determined as a function of $\Omega$, by solving for the zeros of the function~\eqref{eq:ploy2solvem}.  If the order of the SSM $W(E_l)$ is three ($M\!=\!1$), we can solve \eqref{eq:ploy2solvem} for $\rho$ analytically. For higher-order approximations to $W(E_l)$, such analytic solution is unavailable and hence numerical solvers must be used.

\subsection{Stability of the periodic response}

To obtain stability regions of the forced response we apply the Routh-Hurwitz criterion to the Jacobian~\eqref{eq:jacobian}. We conclude that

\begin{eqnarray}
&\mbox{Re}(\lambda_l)+\sum_{m=1}^M (m+1)\mbox{Re}(\beta_m)\rho^{2m}<0,
\label{eqn:stab_cond_A}
\\
&	\frac{a(\rho)}{\rho} \frac{\partial a(\rho)}{\partial \rho} 	
<
\left[\Omega-b(\rho)\right]
\left[ \rho \frac{\partial b(\rho)}{\partial \rho}  -\left(\Omega-b(\rho)\right)\right],
\label{eqn:stab_cond_B}
\end{eqnarray}

\noindent must hold to ensure the asymptotic stability of the forced response. At bifurcations of the response, the inequalities \eqref{eqn:stab_cond_A} and \eqref{eqn:stab_cond_B} become equalities. According to~\eqref{eqn:stab_cond_A}, up to $M$ bifurcation values for $\rho_{crit}$ may arise. These bifurcations appear along straight lines in the $(\rho,\Omega)$ parameter space. From eq.~\eqref{eqn:stab_cond_B}, we obtain that these lines satisfy the equations

\begin{equation}
\Omega_{crit}^{\pm}(\rho)=b(\rho)+\sum_{m=1}^M m\mbox{Im}(\beta_m)\rho^{2m}\pm
\sqrt{\left[ \sum_{m=1}^M m\mbox{Im}(\beta_m)\rho^{2m} \right]^2-\frac{a(\rho)}{\rho} \frac{\partial a(\rho)}{\partial \rho} }.
\label{eq:Omega_crit}
\end{equation}

\noindent These functions divide the $(\rho,\Omega)$ parameter into stable and unstable regions, as indicated in Fig.~\ref{fig:sketch_of_f_res}. If the real parts of the parameters $\beta_m$ are zero or small ($a(\rho)\!\approx\!0$), the graph of $\Omega_{crit}^-$ coincides with the backbone curve (cf.~Fig.~\ref{fig:sketch_of_f_res}).

\begin{figure}
	\begin{center}
		
	\input{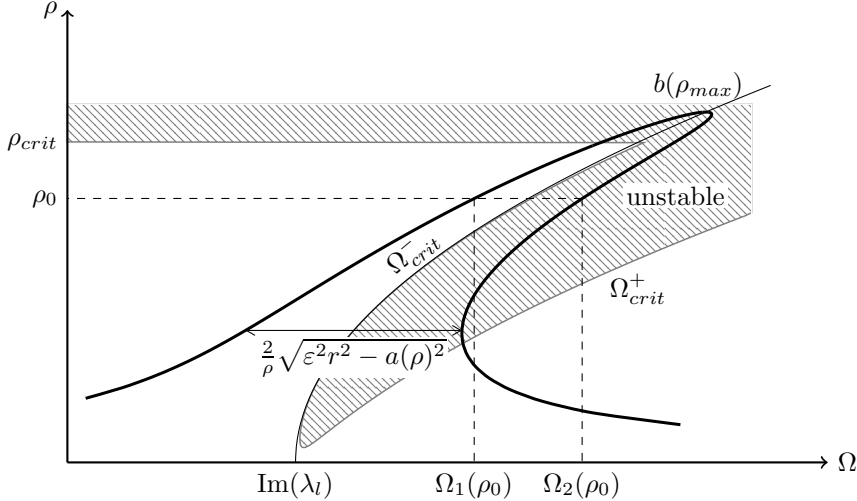}
		
	\end{center}
	\caption{Sketch of a typical forced response for a third-order SSM approximation ($M\!=\!1$) and a nonlinearity with a stiffening effect ($\mbox{Im}(\beta_1)\!>\!0$). }
	\label{fig:sketch_of_f_res}
\end{figure}

\subsection{The periodic response in physical coordinates}

Periodic orbits of~\eqref{eq:dynamics_R_nearres}  are related to periodic orbits in the original  physical coordinates via the parameterization~\eqref{eq:SSM_nearres}. Along the periodic response, the parameterization variable is a complex exponential with amplitude $\rho$ and with the frequency equal to the excitation frequency $\Omega$ (cf. eq.~\eqref{eq:pol_trans}). We insert this exponential into the leading-order expression $\mathbf{W}_0$ for the SSM $W(E_l)$. Since $\mathbf{W}_0$ is a polynomial of $z_l$ and $\overline{z}_l$ (cf. eq.~\eqref{eq:multiindex_exp}), substitution of complex exponential creates higher harmonics ($n\Omega$), whereas the amplitude $\rho$ is exponentiated.

Through the time-varying parameterization $\mathbf{w}_1^{\mathbf{0}}$, terms for the first harmonic  arise (cf.~eq.~\eqref{eq:SSM_nearres_sol}), with their amplitudes given by
\begin{equation}
\mathbf{W}^{\pm}=\mathbf{V S}^{\pm}
(\pm i \Omega\mathbf{I}-\mathbf{\Lambda})^{-1}\mathbf{V}^{-1}
\mathbf{g}^{\pm 1}. 
\end{equation}
With that notation we obtain for the complex amplitudes $\mathbf{x}_{j\Omega}$ of the j-th harmonic of the forcing frequency $\Omega$ as
\begin{equation}
\mathbf{x}_{j\Omega}=
\begin{cases}
\sum_{m=0}^{M} \rho^{2m+j} e^{ji\psi(\rho)}\mathbf{w}_0^{(m+j,m)}
+\delta_{1j}\varepsilon\mathbf{W}^+, 
&
0 \leq j \leq M,
\\[0.2mm]
\sum_{m=0}^{M} \rho^{2m+j} e^{-ji\psi(\rho)}\mathbf{w}_0^{(m,m+j)}
+\delta_{-1j} \varepsilon \mathbf{W}^-, 	  
&
-M \leq j<0,
\end{cases}
\label{eq:backtrans_time}
\end{equation}

\noindent where the coefficients~$\mathbf{w}_0^{(m+j,m)}$ are set to zero, if the corresponding coefficient is higher than the computed order of $\mathbf{W}_0$ ($2m+j\!>\! 2M+1$). From these formulas, one obtains the amplitudes and phases of the response for the fundamental ($|j|\!=\!1$) and superharmonic ($|j|\!>\!1$) frequencies. The case $j\!=\!0$ implies a static shift of the center of the steady state solution, which is a known phenomena for nonlinear system~\eqref{eq:sys_phys} with quadratic stiffness terms~(cf. Nayfeh and Mook~\cite{nayfeh_nonlinear}).

\section{Numerical examples}

We now demonstrate our SSM-based analytic results on forced responses and backbone curves on three numerical examples. The first is a two-degree-of-freedom oscillator introduced by Shaw and Pierre~\cite{SHAW+Pierre}, modified and further studied by Haller and Ponsioen \cite{article:Haller_SSMbasics} and Szalai et al.~\cite{article:SSM_SysID}. The nonlinearity in this oscillator arises from a single cubic spring. Our second example, taken from~Touz{\'e}~and~Amabili~\cite{touze_damped_NNM}, also has two degrees of freedom, but its nonlinearities are more complex, consisting of both quadratic and cubic terms.  To demonstrate the applications of our results to higher-dimensional systems, we select a chain of oscillators with five degrees of freedom for the third example.   

On these three examples, we compare our results with the second-order normal form approach of Neild and Wagg~\cite{Neild_2NF_firstresult} and with a normal-form type method of Touz{\'e}~and~Amabili~\cite{touze_damped_NNM}. Both methods assume that the mechanical system is expressed in modal coordinates and hence the linear part of the system is fully decoupled. 

The Neild-Wagg method introduces a time-dependent transformation to remove forcing terms from all modal coordinates whose eigenfrequencies are not in resonance with the forcing frequency. Afterwards, it identifies the resonant terms in the dynamics via harmonic balance. Two major differences to the present approach are the treatment of damping and the nonlinearities. Specifically, Neild and Wagg~\cite{Neild_2NF_firstresult} assume small nonlinearities and allow only small viscous damping.  Neild~et~al.~\cite{Neild_2ndNF_overview} also add an trivial dynamical equation for the time evaluation of the damping coefficients. Afterwards they carry out the normal form transformations for the enlarged system and obtain that the linear modal damping can be added to the final dynamical equations. Therefore, no transfer of linear damping between the modal coordinates induced by the nonlinearities can be captured. Next, the method employs the harmonic balance to approximate the amplitude of the forced response, which leads to an expression similar to eq.~\eqref{eq:ploy2solvem}. Stability conditions for the steady state solution can be found in~Wagg~and~Neild~\cite{wagg2014nonlinear} and a recent overview in~Neild~et~al.~\cite{Neild_2ndNF_overview}. 

In contrast the Touz{\'e}-Amabili method starts with the unforced and damped mechanical systems in modal coordinates with geometric (position-dependent) nonlinearities.  After a cubic transformation to normal-form type dynamical equations, they restrict their calculations to a subset of coordinates, called the master coordinates. The choice of the master coordinates is motivated heuristically. External forcing is then introduced directly into the normal form, representing simple forcing along non-physical, curvilinear coordinates. In addition, the forcing is assumed to be along the master coordinates only. Therefore, one can only achieve model reduction via this method, if the non-master modal coordinates are unforced, as we highlight in Example 3. We acknowledge the possibility to modify the Touz{\'e}-Amabili method to overcome this shortcoming  by neglecting inconvenient forcing terms. Such a reasoning, however, is not available in the literature and it is beyond the scope of the present study to modify existing methods. We, therefore, follow the method as it is stated in Touz{\'e}~and~Amabili~\cite{touze_damped_NNM}.   In analogy with~Kerschen~et~al.~\cite{kerschen2014modal} and Touz{\'e}~and~Amabili~\cite{touze_damped_NNM} we will obtain the forced response of the reduced dynamics via numerical continuation.  

To compare the accuracy of these two methods to ours, we use the \textsc{Matcont} toolbox~\cite{MATCONT} of~\textsc{Matlab} to calculate the periodic responses in the three examples directly. The result of the continuation are $T$-periodic orbits in the full phase space. As routinely done in the vibrations literature (cf. Kerschen~et~al.~\cite{kerschen2014modal}, Peeters~et~al.~\cite{peeters_backbone}, Neild~et~al.\cite{Neild_2ndNF_overview} and  Touz{\'e}~and~Amabili~\cite{touze_damped_NNM}) the maximal displacement along a modal direction is taken as modal amplitude of the first harmonic. To validate  the formulas for higher harmonics (cf.~\eqref{eq:backtrans_time} for $|j|\!>\!1$), we extract higher harmonics via the Fast Fourier Transformation of selected orbits.

\subsection{Modified Shaw-Pierre example}

Shown in Fig.~\ref{fig:Shaw_Pierre_ex}, this mechanical system was originally introduced by Shaw and Pierre~\cite{SHAW+Pierre}, with modifications appearing in Haller and Ponsioen~\cite{article:Haller_SSMbasics} and Szalai et al.~\cite{article:SSM_SysID}. Its equations of motion are

\begin{equation}
\begin{bmatrix}
m & 0\\
0 & m
\end{bmatrix}
\begin{bmatrix}
\ddot{q}_1\\
\ddot{q}_2
\end{bmatrix}
+
\begin{bmatrix}
c_1 + c_2& -c_2\\
-c_2 & c_1 + c_2
\end{bmatrix}
\begin{bmatrix}
\dot{q}_1\\
\dot{q}_2
\end{bmatrix}
+
\begin{bmatrix}
2k& -k\\
-k& 2k
\end{bmatrix}
\begin{bmatrix}
q_1\\
q_2
\end{bmatrix}
+
\begin{bmatrix}
\kappa q_1^3\\
0
\end{bmatrix}
=
\begin{bmatrix}
f_1\\
f_2
\end{bmatrix}.
\label{eq:SP_sys_eqm}
\end{equation}

\pgfdeclarepatternformonly[\GridSize]{MyGrid}{\pgfqpoint{-1pt}{-1pt}}{\pgfqpoint{4pt}{4pt}}{\pgfqpoint{\GridSize}{\GridSize}}%
{
	\pgfsetlinewidth{0.3pt}
	\pgfpathmoveto{\pgfqpoint{0pt}{0pt}}
	\pgfpathlineto{\pgfqpoint{0pt}{3.1pt}}
	\pgfpathmoveto{\pgfqpoint{0pt}{0pt}}
	\pgfpathlineto{\pgfqpoint{3.1pt}{0pt}}
	\pgfusepath{stroke}
}

\newdimen\GridSize
\tikzset{
	GridSize/.code={\GridSize=#1},
	GridSize=3pt
}

\begin{figure}
	\begin{center}
		
			\begin{tikzpicture}
		
		\tikzstyle{spring}=[thick,decorate,decoration={zigzag,pre length=0.3cm,post
			length=0.3cm,segment length=6}]
		\tikzstyle{damper}=[thick,decoration={markings,  
			mark connection node=dmp,
			mark=at position 0.45 with 
			{
				\node (dmp) [thick,inner sep=0pt,transform shape,rotate=-90,minimum
				width=8pt,minimum height=3pt,draw=none] {};
				\draw [thick] ($(dmp.north east)+(5pt,0)$) -- (dmp.south east) -- (dmp.south
				west) -- ($(dmp.north west)+(5pt,0)$);
				\draw [thick] ($(dmp.north)+(0,-3pt)$) -- ($(dmp.north)+(0,3pt)$);
			}
		}, decorate]
		
		\node[draw,outer sep=0pt,thick] (M1) [minimum width=1.5cm, minimum height=1cm] {} node[above]{$\qquad m$};
		\node[draw,outer sep=0pt,thick] (M2) at (3,0) [minimum width=1.5cm, minimum height=1cm] {} node[above] at (3,0){$\qquad m$};
		\draw[spring] ($(M1.east) - (0,0.25)$) -- ($(M2.west) - (0,0.25)$) 
		node [midway,below] {$k$};
		\draw[damper] ($(M1.east) + (0,0.25)$) -- ($(M2.west) + (0,0.25)$)
		node [midway,above] {\raisebox{0.1cm}{$c_2$}};
		
		
		\node[pattern=north east lines, pattern color=black] at (-2.375,-0.125) (LW) [minimum width=0.25cm,  minimum height=1.75cm] {};
		
		\node[pattern=north east lines, pattern color=black] at (1.5,-0.875) (BW) [minimum width=7.5cm,  minimum height=0.25cm] {};
		
		\node[pattern=north east lines, pattern color=black] at (5.375,-0.125) (RW) [minimum width=0.25cm,  minimum height=1.75cm] {};
		
		\draw[spring] ($(LW.east) - (0,0.125)$) -- ($(M1.west) - (0,0.25)$) node [midway,below] {$k,\kappa$};
		\draw[damper] ($(LW.east) + (0,0.375)$) -- ($(M1.west) + (0,0.25)$)
		node [midway,above] {\raisebox{0.1cm}{$c_1$}};
		
		\draw[spring] ($(M2.east) - (0,0.25)$) -- ($(RW.west) - (0,0.125)$) node [midway,below] {$k$};
		\draw[damper] ($(M2.east) + (0,0.25)$) -- ($(RW.west) + (0,0.375)$)
		node [midway,above] {\raisebox{0.1cm}{$c_1$}};
		
		\draw (-2.25,0.75)--(-2.25,-0.75)--(5.25,-0.75)--(5.25,0.75);
		
		\draw[thick] (-0.5,-0.625) circle (0.125);
		\draw[thick] (0.5,-0.625) circle (0.125);
		\draw[thick] (2.5,-0.625) circle (0.125);
		\draw[thick] (3.5,-0.625) circle (0.125);
		
		\draw[-triangle 45] (-0.5,0)--(0.5,0) node[midway,below]{$f_1$};
		\draw[-triangle 45] (2.5,0)--(3.5,0) node[midway,below]{$f_2$};
		
		\draw[thick,->] (-1.95,-0.4)-- (-1.15,-0.05) ;
		\draw[-latex] (-0.5,0.6)--(0.5,0.6) node[midway,above]{$q_1$};
		\draw[-latex] (2.5,0.6)--(3.5,0.6) node[midway,above]{$q_2$};
		\end{tikzpicture}
		
	\end{center}
	\caption{The modified Shaw-Pierre example discussed in~\cite{article:SSM_SysID}. We select the nondimensional parameters $m\!=\!1$, $k\!=\!1$, $c_1\!=\!\sqrt{3}c_2\!=\!0.003$ and $\kappa\!=\!0.5$. } 
	\label{fig:Shaw_Pierre_ex}
\end{figure}
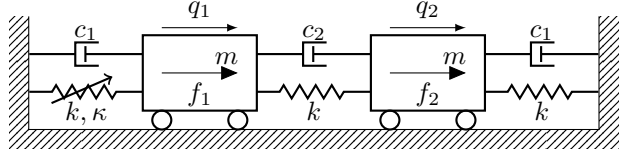

\noindent The system is of the general form~\eqref{eq:sys_phys} and hence the approach developed here applies. The eigenvalues and mode shapes of the linearized dynamics at $q_1\!=\!q_2\!=\!0$ are

\begin{equation*}
\begin{split}
\begin{array}{l}
\lambda_{1,3}=-D_1\omega_1\pm i \omega_1\sqrt{1-D_1^2}\\
\lambda_{2,4}=-D_2\omega_2\pm i \omega_2\sqrt{1-D_2^2}\\
\end{array}
,
\quad
\begin{array}{l}
\omega_1=\sqrt{\frac{k}{m}}, \\
\omega_2=\sqrt{\frac{3k}{m}},
\end{array}
\quad 
\begin{array}{l}
D_1=\frac{c_1}{2\sqrt{km}},\\
D_2=\frac{c_1+2c_2}{\sqrt{12km}},
\end{array}
\quad
\begin{array}{l}
\mathbf{e}_1=
\frac{1}{\sqrt{2}}
\begin{bmatrix}
1&
1
\end{bmatrix}^T
,\\
\mathbf{e}_2=
\frac{1}{\sqrt{2}}
\begin{bmatrix}
1 &
-1
\end{bmatrix}^T.
\end{array}
\end{split}
\end{equation*}

\noindent  For sufficiently small damping the  strengthened non-resonance conditions~\eqref{eq:ext_nearres_cond} hold. By choosing~\mbox{$c_1\!=\!\sqrt{3}c_2$}, the conditions~\eqref{eq:nonres_cond} are satisfied and hence two non-autonomous SSMs exist. The unforced limit of these SSMs and their reduced dynamics have already been calculated by Szalai~et~al.~\cite{article:SSM_SysID}. 

To apply the Touz{\'e}-Amabili method, we have to assume forcing along one of the modal coordinates only. First, we investigate forcing along the first modal coordinate ($f_1\!=\!f_2\!=\!\varepsilon/\sqrt{2}\cos(\Omega t)$) with the amplitude $\varepsilon\!=\!0.003$.  We plot the first and third  harmonics of the first modal amplitude ($p_{1,1\Omega}$ and $p_{1,3\Omega}$) in Fig.~\ref{fig:SP_modal_ampl1}. For comparison, we show the results obtained from the Neild-Wagg method, the Touz{\'e}-Amabili method and numerical continuation with the \textsc{Matlab} toolbox \textsc{Matcont}~\cite{MATCONT} in Fig.~\ref{fig:SP_modalampl1_harm1}, with the later serving as a benchmark to hit. We indicate unstable periodic orbits in dashed lines.

\begin{figure}[hbp!]
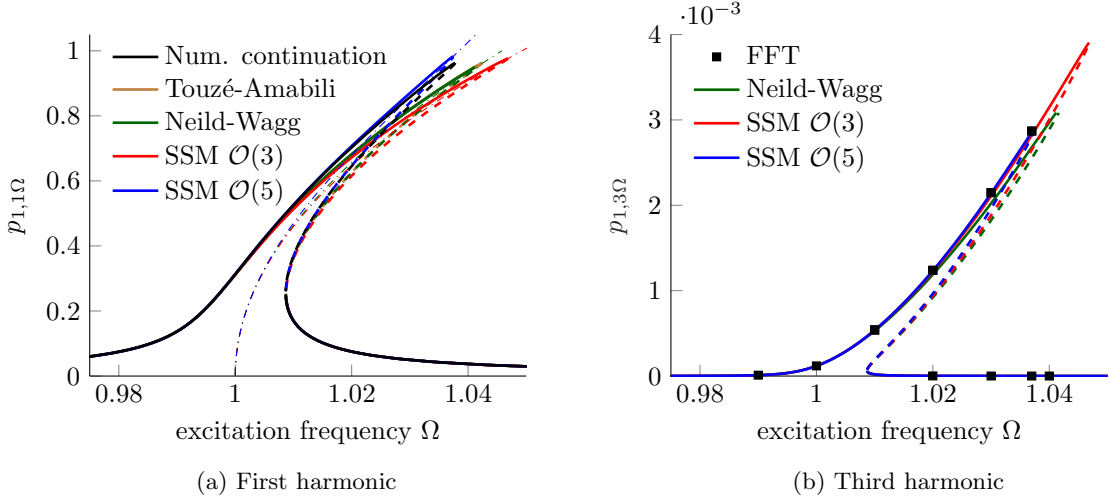

	\begin{center}
		\begin{subfigure}[b]{0.49\textwidth}
				\input{Pics/SP_p1_harm1.tex}
			\caption{First harmonic}
			\label{fig:SP_modalampl1_harm1}
		\end{subfigure}
		\begin{subfigure}[b]{0.49\textwidth}
				\input{Pics/SP_p1_harm3.tex}
			\caption{Third harmonic}
			\label{fig:SP_modalampl1_harm3}
		\end{subfigure}
	\end{center}
	\caption{Response curves (thick lines) and backbone curves (thin dash-dotted lines) for Example~1~under forcing of the first mode; solid lines imply stability and dashed lines instability of the forced response; parameters: $m\!=\!1$, $k\!=\!1$, $c_1\!=\!\sqrt{3}c_2\!=\!0.003$, $\kappa\!=\!0.5$, $f_1\!=\!f_2\!=\!\varepsilon/\sqrt{2}\cos(\Omega t)$ and $\varepsilon\!=\!0.003$. }
	\label{fig:SP_modal_ampl1}
\end{figure}

While all three methods give results close to the numerical continuation, the $\mathcal{O}(5)$ SSM is the most accurate. This approach however, is of higher order than the others. The Touz{\'e}-Amabili method can in principle be extended to higher orders in the coordinates but it assumes modal forcing. To improve the results of the Neild-Wagg method, one would also need to include higher-order terms in their perturbation  approach, which would complicate the calculations significantly. To our best knowledge, higher-order estimates have only been obtained for one-degree of freedom oscillators (cf.~Neild and Wagg~\cite{neild_general_freq_detun,wagg2014nonlinear}). Out of all third-order methods, the $\mathcal{O}(3)$ SSM computation gives the weakest result.

Touz{\'e} and Amabili~\cite{touze_damped_NNM} do not explicitly estimate the amplitudes of higher harmonics of the forced response of system~\eqref{eq:sys_phys}, hence the omission of the results from their method in Fig.~\ref{fig:SP_modalampl1_harm3}.  Note, that a periodic solution to their reduced dynamics contains fundamental and higher harmonics, which could be related to amplitudes in physical coordinates via their normal form transformation. Here, however, we follow the published results of Touz{\'e} and Amabili~\cite{touze_damped_NNM} without modifications.

The reference solution is generated by the Fast Fourier Transformation (FFT) of the continuation signal. Again, all four result agree closely, but the $\mathcal{O}(5)$ SSM method matches the reference solution the best. Due to the cubic nonlinearity, the modal amplitudes at even harmonics are zero for the accuracy investigated in this article.  

Next, we apply forcing along the second modal degree of freedom ($l\!=\!2$), by selecting \linebreak $ f_1\!=\!-f_2\!=\!\varepsilon/\sqrt{2}\cos(\Omega t) $ and $\varepsilon\!=\!0.01$. We show the  first and third harmonics of the computed forced response in Fig.~\ref{fig:SP_modal_ampl2}.  Again, the $\mathcal{O}(5)$ SSM approach approximates the benchmark solution most accurately. The results from the other methods are nearby and align closely with each other.

\begin{figure}[ht!]
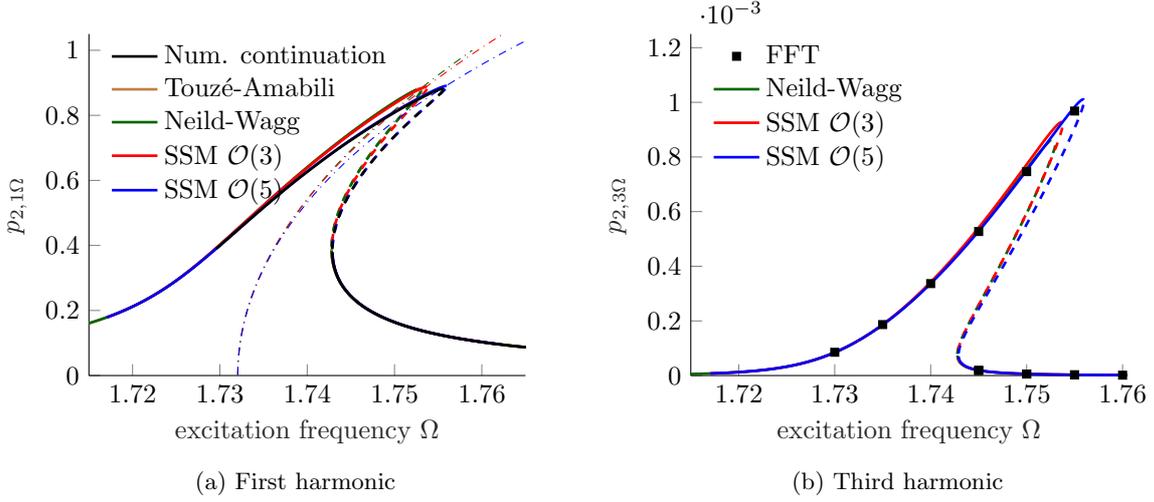

	\begin{center}
		\begin{subfigure}[b]{0.49\textwidth}
				\input{Pics/SP_p2_harm1.tex}
			\caption{First harmonic}
			\label{fig:SP_modalampl2_harm1}
		\end{subfigure}
		\begin{subfigure}[b]{0.49\textwidth}
				\input{Pics/SP_p2_harm3.tex}
			\caption{Third harmonic}
			\label{fig:SP_modalampl2_harm3}
		\end{subfigure}
	\end{center}
	\caption{Response curves (thick lines) and backbone curves (thin dash-dotted lines) for Example 1 under forcing of the second mode;  solid lines imply stability and dashed lines instability of the forced response; parameters:$m\!=\!1$, $k\!=\!1$, $c_1\!=\!\sqrt{3}c_2\!=\!0.003$, $\kappa\!=\!0.5$, $f_1\!=\!-f_2\!=\!\varepsilon/\sqrt{2}\cos(\Omega t)$ and $\varepsilon\!=\!0.01$.}
	\label{fig:SP_modal_ampl2}
\end{figure}

\subsection{Spring system}

Our second example involves a mass suspended via a vertical and a horizontal spring to the wall (cf.~Fig.~\ref{fig:Spring_sys_plot}). Touz{\'e}~et~al.~\cite{touze_springsys_cons} derive the equation of motion for this system up to third order. With viscous damping and nondimensional parameters, the equations of motion subject to horizontal forcing are 

\begin{equation}
\begin{split}
\ddot{q}_1+2D_1\omega_1\dot{q}_1+\omega_1^2q_1+
\frac{\omega_1^2}{2}(3q_1^2+q_2^2)+\omega_2^2q_1q_2+\frac{\omega_1^2+\omega_2^2}{2}q_1(q_1^2+q_2^2)&=F_1=f_1\cos(\Omega t),\\
\ddot{q}_2+2D_2\omega_2\dot{q}_2+\omega_2^2q_2+
\frac{\omega_2^2}{2}(3q_2^2+q_1^2)+\omega_1^2q_1q_2+\frac{\omega_1^2+\omega_2^2}{2}q_2(q_1^2+q_2^2)&=0.
\end{split}
\label{eq:spring_sys}
\end{equation}

\noindent We choose the form of the damping and forcing for a direct comparison with Touz{\'e} and Amabili~\cite{touze_damped_NNM}, but the theory developed here also applies to nonlinear damping and general forcing.

\begin{figure}[hbp!]
	\begin{center}
		
		\begin{tikzpicture}
		
		\tikzstyle{spring}=[thick,decorate,decoration={zigzag , pre length=0.6cm,post
			length=0.6cm,segment length=6}]
		\tikzstyle{springd}=[thick,decorate,dashed,decoration={zigzag , pre length=0.6cm,post
			length=0.6cm,segment length=6}]
		\tikzstyle{damper}=[thick,decoration={markings,  
			mark connection node=dmp,
			mark=at position 0.45 with 
			{
				\node (dmp) [thick,inner sep=0pt,transform shape,rotate=-90,minimum
				width=8pt,minimum height=3pt,draw=none] {};
				\draw [thick] ($(dmp.north east)+(5pt,0)$) -- (dmp.south east) -- (dmp.south
				west) -- ($(dmp.north west)+(5pt,0)$);
				\draw [thick] ($(dmp.north)+(0,-3pt)$) -- ($(dmp.north)+(0,3pt)$);
			}
		}, decorate]
		
		
		\node[draw,inner sep=0pt,thick,circle] at (2,2.5) (M)  [minimum size=0.5cm] {$m$};

		\node[pattern=north east lines, pattern color=black] at (-2.125,0.75) (LW) [minimum width=0.25cm,  minimum height=5cm] {};
		
		\node[pattern=north east lines, pattern color=black] at (0.25,-1.625) (BW) [minimum width=5cm,  minimum height=0.25cm] {};
		
		\draw[spring] ($(M.west) $) -- ($(LW.east) + (0,1)$) node [midway,above] {$k_1$} ;
		
		\draw[spring] ($(M.south) $) -- ($(BW.north) + (1,0.0)$) node [midway,right] {$k_2$};

		\node[draw,inner sep=0pt,thick,circle,dashed] at (1.375,1.875) (Md)  [minimum size=0.5cm] { };
		\draw[springd] ($(Md.west) $) -- ($(LW.east) + (0,1)$);
		
		\draw[springd] ($(Md.south) $) -- ($(BW.north) + (1,0.0)$) ;
		\draw[-triangle 45] ($(M.east) $)--($(M.east) +(1,0)$) node[midway,above]{$F_1$};	
		
		\draw (-2,-1.5)-- (2.75,-1.5); 
		\draw  (-2,-1.5)-- (-2,3.25); 
		
		\draw[-latex] (-1.5,-1.1)--(-1.5,-0.5) node[anchor=east]{$q_2$};
		\draw[-latex] (-1.6,-1)--(-1,-1) node[anchor=north]{$q_1$};
		
		

		\end{tikzpicture}
		
	\end{center}
	\caption{The mechanical system in Example 2, discussed by Touz{\'e} and Amabili~\cite{touze_damped_NNM}. }
	\label{fig:spring_sys}
\end{figure}
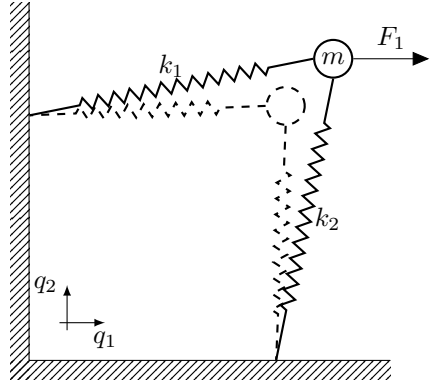

We select the parameters $\omega_1\!=\!2$, $\omega_2\!=\!4.5$, $D_1\!=\!0.01$, $D_2\!=\!0.2$ and $f_1$ is set to $0.02$. In Fig.~\ref{fig:Spring_sys_plot}, we plot the amplitude of the coordinate $q_1$ at the first harmonic. Again, the results from numerical continuation serve as the benchmark solution. The SSMs of order three and five and the results from the Touz{\'e}-Amabili method agree well and show the same qualitative behavior as the benchmark solution. The Neild-Wagg method  incorrectly predicts hardening  behavior of the backbone curve and overestimates the amplitude. The latter arises because of the treatment of the damping by the Neild-Wagg method. In their method, modal damping is  added directly to the normal form (cf.~Neild~et~al.~\cite{Neild_2ndNF_overview}) and no transfer of the linear modal damping via nonlinearities arises. Without referring to this specific method, Touz{\'e} and Amabili~\cite{touze_damped_NNM} point out this issue for another method that incorporates damping in a similar manner. The incorrect bending behavior arises due to the assumption of small nonlinearities, based on which all quadratic terms of nonlinearities are neglected for the backbone curve estimation in the Neild-Wagg method. To recover the effect of quadratic nonlinearities on the backbone curve, a higher-order extension in their perturbation approach is required, which would complicate the calculations significantly and is unavailable in the literature at this time. In summary, the small damping and small nonlinearity assumption made in the Neild-Wagg method is, therefore, not justified for this example.

\begin{figure}[ht!]
	\begin{center}
			\input{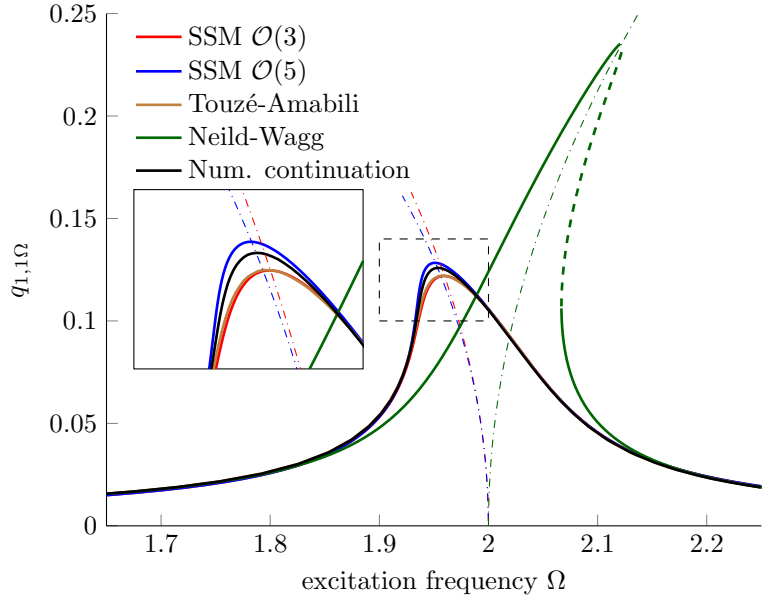}
	\end{center}
	\caption{ Amplitude of the first harmonic of the displacement $q_1$ (thick lines) and backbone curves (thin dash-dotted lines) ; solid lines imply stability and dashed lines instability of the forced response; parameters:  $\omega_1\!=\!2$, $\omega_2\!=\!4.5$, $D_1\!=\!0.01$, $D_2\!=\!0.2$ and $f_1=0.02$.}
	\label{fig:Spring_sys_plot}
\end{figure}

Enlarging the frequency response curve for high amplitudes (inset in Fig.~\ref{fig:Spring_sys_plot}), we observe that the $\mathcal{O}(5)$ SSM construction shows the highest accuracy again. The results from the $\mathcal{O}(3)$ SSM and those from the Touz{\'e}-Amabili method almost coincide.

\subsection{Oscillator chain}

As an application to a higher-dimensional system, we extend Example 1 (cf.~Fig.~\ref{fig:Shaw_Pierre_ex}) into a chain of coupled oscillators. The first and the last mass are suspended to the wall, as shown in Fig.~\ref{fig:Chain}. We assume a cubic nonlinearity for the spring suspending the first mass to the wall. 

\begin{figure}[hb!]
	\begin{center}
		
		\begin{tikzpicture}
		
		\tikzstyle{spring}=[thick,decorate,decoration={zigzag,pre length=0.3cm,post
			length=0.3cm,segment length=6}]
		\tikzstyle{damper}=[thick,decoration={markings,  
			mark connection node=dmp,
			mark=at position 0.45 with 
			{
				\node (dmp) [thick,inner sep=0pt,transform shape,rotate=-90,minimum
				width=8pt,minimum height=3pt,draw=none] {};
				\draw [thick] ($(dmp.north east)+(5pt,0)$) -- (dmp.south east) -- (dmp.south
				west) -- ($(dmp.north west)+(5pt,0)$);
				\draw [thick] ($(dmp.north)+(0,-3pt)$) -- ($(dmp.north)+(0,3pt)$);
			}
		}, decorate]

		
		\node[pattern=north east lines, pattern color=black] at (-2.375,-0.125) (LW) [minimum width=0.25cm,  minimum height=1.75cm] {};
		
		\draw[spring] ($(LW.east) - (0,0.125)$) -- ($(M1.west) - (0,0.25)$) node [midway,below] {$k,\kappa$};
		\draw[damper] ($(LW.east) + (0,0.375)$) -- ($(M1.west) + (0,0.25)$)
		node [midway,above] {\raisebox{0.1cm}{$c$}};
		\draw[thick,->] (-1.95,-0.4)-- (-1.15,-0.05) ;
		
		\node[draw,outer sep=0pt,thick] (M1) [minimum width=1.5cm, minimum height=1cm] {} node[above]{$\qquad m$};
		\draw[thick] ($(M1.south) + (-0.5,-0.125)$) circle (0.125);
		\draw[thick] ($(M1.south) + (0.5,-0.125)$) circle (0.125);
		\draw[-triangle 45] ($(M1) - (0.5,0)$) -- ($(M1) + (0.5,0)$)  node[midway,below]{$f_1$};
		\draw[-latex] ($(M1.north) + (-0.5,0.1)$) -- ($(M1.north) + (0.5,0.1)$) node[midway,above]{$q_1$};		
		
		\coordinate (LBs) at (2.25,-0.25);	
		\coordinate (LBd) at (2.45,0.25);	
		\draw[black,line width=0.3mm] plot [smooth] coordinates {(2.35,-0.6) (LBs) (LBd)  (2.35,0.6)} ;
		\draw[spring] ($(M1.east) - (0,0.25)$) -- (LBs) 
		node [midway,below] {$k$};
		
		\draw[damper] ($(M1.east) + (0,0.25)$) -- (LBd)
		node [midway,above] {\raisebox{0.1cm}{$c$}};
		
		\node[above](dots) at ($(LBs) + (0.3,0.125)$) {...};
		
		\coordinate (RBs) at (2.65,-0.25);	
		\coordinate (RBd) at (2.85,0.25);	
		\draw[black,line width=0.3mm] plot [smooth] coordinates {(2.75,-0.6) (RBs) (RBd)  (2.75,0.6)} ;

		\node[draw,outer sep=0pt,thick] (M2) at (4.9,0) [minimum width=1.5cm, minimum height=1cm] {} node[above] at (4.9,0){$\qquad m$};

		\draw[spring] (RBs)--($(M2.west) - (0,0.25)$)   
		node [midway,below] {$k$};
		\draw[damper] (RBd) -- ($(M2.west) + (0,0.25)$) 
		node [midway,above] {\raisebox{0.1cm}{$c$}};

		\draw[thick] ($(M2.south) + (-0.5,-0.125)$) circle (0.125);
		\draw[thick] ($(M2.south) + (0.5,-0.125)$) circle (0.125);
		\draw[-triangle 45] ($(M2) - (0.5,0)$) -- ($(M2) + (0.5,0)$)  node[midway,below]{$f_N$};	
		\draw[-latex] ($(M2.north) + (-0.5,0.1)$) -- ($(M2.north) + (0.5,0.1)$) node[midway,above]{$q_N$};
		
		\node[pattern=north east lines, pattern color=black] at (7.275,-0.125) (RW) [minimum width=0.25cm,  minimum height=1.75cm] {};
		
		\draw[spring] ($(M2.east) - (0,0.25)$) -- ($(RW.west) - (0,0.125)$) node [midway,below] {$k$};
		\draw[damper] ($(M2.east) + (0,0.25)$) -- ($(RW.west) + (0,0.375)$)
		node [midway,above] {\raisebox{0.1cm}{$c$}};

		\draw (-2.25,0.75)--(-2.25,-0.75)--(7.15,-0.75)--(7.15,0.75);

		\node[pattern=north east lines, pattern color=black] at (2.45,-0.875) (BW) [minimum width=9.4cm,  minimum height=0.25cm] {};

		\end{tikzpicture}
		
	\end{center}
	\caption{The oscillator chain in Example 3.}
	\label{fig:Chain}
\end{figure}
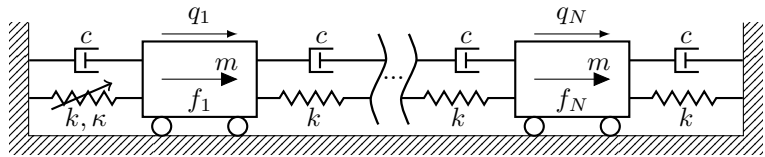

The equation of motion of the $j$-th oscillator, pictured in Fig.~\ref{fig:Chain}, is 

\begin{equation}
m\ddot{q}_j+c(2\dot{q}_{j}-\dot{q}_{j+1}-\dot{q}_{j-1})+k(2q_{j}-q_{j+1}-q_{j-1})+\delta_{1j}\kappa q_{j}^3=f_j(t),
\qquad
j=1,...,N.
\label{eq:eqm_chain}
\end{equation}

\noindent We consider the configuration $m\!=\!1$, $c\!=\!0.005$, $k\!=\!1$ and $\kappa\!=\!0.5$, with the number of degrees of freedom set to $N\!=\!5$. For this choice of the parameters, the natural frequencies and modal damping values are  
\begin{equation}
\begin{array}{l| c c c c c }
j & 1 & 2 & 3 & 4 & 5\\
\hline
\omega_j^2 & 2-\sqrt{3} & 1 & 2 & 3 &  2+\sqrt{3} 
\end{array},
\qquad
D_j=\frac{c}{2},
\qquad
j=1,...,N.
\end{equation}

The non-resonance conditions \eqref{eq:nonres_cond} and \eqref{eq:ext_nearres_cond} are satisfied for $l\!=\!1$ and hence the  two-dimensional time-varying SSM exists. We assume forcing at the first mass~ ($f_1=\varepsilon\cos(\Omega t)$ and  $f_j=0~~j=2,...,N$). The frequency is chosen to be close to the lowest eigenfrequency ($\omega_1\approx0.518$) and the amplitude is set to be $\varepsilon\!=\!0.004$. The amplitudes of the first harmonic of the fifth coordinate $q_{5,1\Omega}$ are plotted in Fig.~\ref{fig:Spring_sys_plot}.

\begin{figure}[hb!]
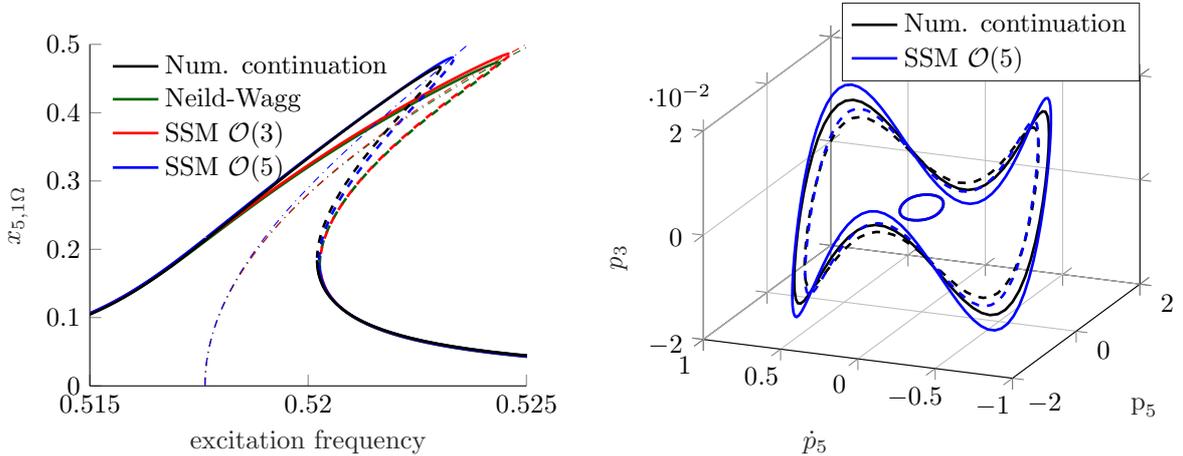

	\begin{center}
		\subcaptionbox{Amplitude of the first harmonic of the displacement of the fifth mass and backbone curvesfor forcing frequencies in the vicinity of the lowest eigenfrequency ($\omega_1 \approx 0.518$).
			\label{fig:Chain_N5_harm1}}%
		[0.48\textwidth]{
			\input{Pics/Chain_N5_X5_harm1.tex}
		}
		\hspace{0.2cm}		 
		\subcaptionbox{Periodic orbits (two stable, one unstable) for the excitation frequency $\Omega\!=\!0.522$, projected in the modal subspace spanned by $p_5$, $\dot{p}_5$ and $p_3$ 
			\label{fig:Chain_N5_orbits}}%
		[0.48\textwidth]{
				\input{Pics/Chain_N5_orbits.tex}
		}		 
	\end{center}
	\caption{Results for Example 3 with $N\!=\!5$ degrees of freedom;  solid lines imply stability and dashed lines instability of the forced response; parameters: $m\!=\!1$, $k\!=\!1$, $c\!=\!0.005$, $\kappa\!=\!0.5$, $f_j=\delta_{1j}\varepsilon\cos(\Omega t)$ and $\varepsilon\!=\!0.004$. }
	\label{fig:Chain_N5_orbit}
\end{figure}

Since the forcing is not aligned with a set of modal coordinates the unmodified Touz{\'e}-Amabili method is inapplicable in this example, and hence will be omitted in our comparison. Again, the $\mathcal{O}(5)$ SSM matches most accurately with the benchmark solution, obtained via numerical continuation. The frequency response curves from the SSM $\mathcal{O}(3)$ and the Neild-Wagg method deviate from the benchmark solution. 

Even specific orbits computed from the analytic SSM expression match closely with the orbits obtained by numerical continuation (cf. Fig. \ref{fig:Chain_N5_orbits}).

\section{Conclusion}

We have derived highly accurate analytic expressions for the forced response and backbone curves of damped and forced nonlinear mechanical systems of arbitrary dimension. Our procedure constructs an approximation for the two-dimensional, non-autonomous spectral submanifolds (SSMs) that act as nonlinear continuations of modal subspaces of the linearized system.  The existence, uniqueness and smoothness of the SSMs are guaranteed under low-order non-resonance condition on the eigenvalues of the linearization (cf. Haller and Ponsioen~\cite{article:Haller_SSMbasics}). We establish that a given autonomous (time independent) SSM can be continued for the externally forced system, unless the forcing is in resonance with an imaginary part of an eigenvalue of the linearized system.

For backbone curve calculations, we focus on such resonant external forcing and construct a two-dimensional non-autonomous SSM. Constructing the SSM via the parametrization method (cf.~Cabre et al.~\cite{Cabre_para}), the reduced dynamics is simplified significantly. Due to this simplification, we are able to derive a polynomial expression whose zeros determine the amplitudes of the forced response. Furthermore, we calculate backbone curves, stability regions and discover a symmetry of the forced response. 

With our analytical treatment, we confirm the $90$-degree phase lag criterion of the response at the backbone from Peeters et al.~\cite{peeters_backbone}  for any damping that is a polynomial function of the velocities and positions. Furthermore, we connect the backbone curve, directly computed from the forced response, to the frequency-amplitude relationship of the conservative unforced limit of the system.    

We demonstrate the performance of our results on three numerical examples. Comparing with the Neild-Wagg method~\cite{Neild_2NF_firstresult},  the Touz{\'e}-Amabili method~\cite{touze_damped_NNM} and numerical continuation, we conclude an overall superior performance for the $\mathcal{O}(5)$ SSM approach. While our method is applicable to general single harmonic external forcing, a model-order reduction with the Touz{\'e}-Amabili method is only achievable for forcing along a modal direction. Investigating Example 2, where the nonlinearities are of quadratic and cubic form, we discover incorrect predictions from the Neild-Wagg method.

Ponsioen et al.~\cite{Sten_auto_SSM} describe an automated computational algorithm to approximate two-dimensional SSMs of nonlinear mechanical systems up to arbitrary order. To increase the precision of our results further, these high-order approximations, can be coupled with the results of this article, which is our ongoing effort.

We  have limited our discussion to two-dimensional SSMs. For multi-frequency forcing in resonance with multiple eigenvalues of the linearized system, or for structures whose eigenfrequencies are integer multiples of each other (i.e., eq.~\eqref{eq:ext_nearres_cond} is violated), a reduction to a higher-dimensional SSM is necessary. Since  the theory developed by Cabre et al.~\cite{Cabre_para} and Haro and de la Llave~\cite{HARO_SSM} applies to higher-dimensional submanifolds, our calculations can be extended to the multi-frequency setting.

\section{Acknowledgments}

We are grateful to Florian Kogelbauer for helpful discussions and for pointing out an error in an earlier stage of this work. We are also thankful to Cyril Touz{\'e} for helpful suggestions and clarifications on his method.

\setcounter{section}{0}
\renewcommand\thesection{\Alph{section}}
\renewcommand{\thesubsection}{\thesection\arabic{subsection}}
\section{Derivations}

\subsection{Derivations for the SSMs of the forced system}
\label{app:der_res1}
The results stated in  \eqref{eq:w1_harm} and \eqref{eq:r1_harm} follow from the work of Haller and Ponsionen~\cite{article:Haller_SSMbasics} or Haro et al.~\cite{HARO_SSM}. Both results are generally applicable to system~\eqref{eq:sys_phys}, since the trivial fixed point of the unforced equation \eqref{eq:sys_phys} is hyperbolic (cf. conditions~\eqref{eq:cond_eigval}). Since we assume that the non-resonance conditions~\eqref{eq:nonres_cond} hold, an SSM associated with the modal subspace~$E$ uniquely exists and persists for small~$\varepsilon\! \geq\! 0$. Haro and de la Llave~\cite{HARO_SSM} formulate their main results for discrete mappings, but also show the direct applicability of their results to flow maps of continuous systems. 

To calculate a parameterization $\mathbf{W}(\mathbf{z},\phi)$ and the associated reduced dynamics $\mathbf{R}(\mathbf{z},\phi)$, we start from the invariance condition~\eqref{eq:inv_cond_SSM} in which we substitue the series expansion~\eqref{eq:expans_SSM+dyn}. With the notation~\eqref{eq:multiindex_exp}, the series expansion~\eqref{eq:expans_SSM+dyn} can be rewritten as 
	\begin{equation}
	\begin{split}
	\mathbf{W}(\mathbf{z},\mathbf{\phi})&=
	\mathbf{W}_0(\mathbf{z})+\varepsilon \mathbf{W}_1(\mathbf{z},\mathbf{\phi}) 
	+\mathcal{O}(\varepsilon^2)
	=
	\mathbf{W}_0(\mathbf{z})+\varepsilon \mathbf{w}^{\mathbf{0}}_1(\mathbf{\phi}) 
	+\mathcal{O}(\varepsilon|\mathbf{z}|,\varepsilon^2),
	\\
	\mathbf{R}(\mathbf{z},\mathbf{\phi})&=
	\mathbf{R}_0(\mathbf{z})+\varepsilon \mathbf{R}_1(\mathbf{z},\mathbf{\phi}) 
	+\mathcal{O}(\varepsilon^2)
	=
	\mathbf{R}_0(\mathbf{z})+\varepsilon \mathbf{r}^{\mathbf{0}}(\mathbf{\phi}) 
	+\mathcal{O}(\varepsilon|\mathbf{z}|,\varepsilon^2). 
	\end{split}
	\label{eq:expans_SSM+dyn_detail}
	\end{equation}
	Substituting the expansion~\eqref{eq:expans_SSM+dyn_detail} in the invariance condition~\eqref{eq:inv_cond_SSM}, we obtain

\begin{equation}
\begin{split}
&\mathbf{A}
(\mathbf{W}_0(\mathbf{z})+\varepsilon \mathbf{W}_1(\mathbf{z},\mathbf{\phi}))
+
\mathbf{G}_{nlin}(\mathbf{W}_0(\mathbf{z})+\varepsilon \mathbf{W}_1(\mathbf{z},\mathbf{\phi}))+
\varepsilon\mathbf{G}_{ext}(\mathbf{\phi})\\
&=
\frac{\partial \mathbf{W}_0(\mathbf{z}) }{\partial \mathbf{z}}\mathbf{R}_0(\mathbf{z})
+\varepsilon
\left(
\frac{\partial \mathbf{W}_1(\mathbf{z},\mathbf{\phi}) }{\partial \mathbf{z}}\mathbf{R}_0(\mathbf{z})
+
\frac{\partial \mathbf{W}_0(\mathbf{z}) }{\partial \mathbf{z}}\mathbf{R}_1(\mathbf{z},\mathbf{\phi})
+
\frac{\partial \mathbf{W}_1(\mathbf{z}) }{\partial \mathbf{\phi}}\Omega
\right)+\mathcal{O}(\varepsilon^2).
\end{split}
\label{eq:inv_cond_exp}
\end{equation}

\noindent Comparing equal orders of $\varepsilon$, we find, that the zeroth order part of \eqref{eq:inv_cond_exp} does not contain forcing terms

\begin{equation}
\mathcal{O}(\varepsilon^0):
\qquad
\mathbf{A}\mathbf{ W_0}+\mathbf{G}_{nlin}(\mathbf{W_0})
=
\frac{\partial \mathbf{W}_0}{\partial  \mathbf{z}}\mathbf{R}_0.
\label{eq:inv_first_order}
\end{equation}

Since the non-resonance conditions~\eqref{eq:nonres_cond} are satisfied, one can solve for the unknown polynomial coefficients of $\mathbf{W}_0$ and $\mathbf{R}_0$ (cf.~Haro and de la Llave.~\cite{HARO_SSM}). The fixed point is at the origin, therefore 
\begin{equation}
\mathbf{w}_0^{\mathbf{0}}=0,\qquad  \mathbf{r}_0^{\mathbf{0}}=0,
\label{eq:ROW}
\end{equation} 
holds.

Now, we consider the first-order terms in epsilon and the zeroth-order in $\mathbf{z}$ in eq.~\eqref{eq:inv_cond_exp}. At this order, no terms from the nonlinearity  $\mathbf{G}_{nlin}$ arise, since it is at least quadratic in its arguments. The relevant terms at this order from the right hand side of eq.~\eqref{eq:inv_cond_exp}  $\mathbf{W}_0$ and  $\mathbf{R}_0$ are zero (cf. eq.~\eqref{eq:ROW}), while the terms from  $\mathbf{W}_1$ and $\mathbf{R}_1$ remain

\begin{equation}
\mathcal{O}(\varepsilon,|\mathbf{z}|^0):
\qquad
\mathbf{A}\mathbf{w}_1^{\mathbf{0}}(\phi)
+\mathbf{G}_{ext} (\phi)
=
\frac{\partial \mathbf{W}_0}{\partial  \mathbf{z}}\mathbf{r}_1^{\mathbf{0}}(\phi)
+\frac{\partial \mathbf{w}_1^{\mathbf{0}}}{\partial  \mathbf{\phi}}\Omega.
\label{eq:inv_cond_e1z0}
\end{equation}

Setting $\mathbf{r}_1^{\mathbf{0}}(\phi)=0$, we find that the periodic solutions of \eqref{eq:inv_cond_e1z0} can be obtained from Dunhamel's principle as

\begin{equation}
\mathbf{w}_{1}^{\mathbf{0}}=\int_{0}^t e^{\mathbf{A}(t-s)}\mathbf{G}_{ext}(s)ds,
\label{eq:dunhamel_sol}
\end{equation}

\noindent where we use the time evolution~\eqref{eq:angles} of the angles. The integral in~\eqref{eq:dunhamel_sol} can be evaluated numerically or solved analytically. In our setting, the external forcing can be expanded in a Fourier series leading to the expression~\eqref{eq:w1_harm}.

Again, the existence and smoothness properties stated by Haro and de la Llave~\cite{HARO_SSM} ensure the solvability of the higher-order terms in $\mathbf{z}$ and $\varepsilon$ in~\eqref{eq:inv_cond_exp}.

\subsection{SSM for the near-resonant forced mechanical system}
\label{app:der_res2}

To prove Theorem~\ref{Thm:auto_SSM} we explicitly calculate $\mathbf{W}(\mathbf{z},\phi)$ and $\mathbf{R}(\mathbf{z},\phi)$ up to the required order of accuracy $\mathcal{O}(\mu^{2M+3})$.  We first identify the relevant terms to be calculated by applying the rescaling~\eqref{eq:rescaling} to the series expansion~\eqref{eq:expans_SSM+dyn}: 
\begin{equation}
\begin{split}
\mathbf{W}(\mathbf{z},\phi)&=\mathbf{W}_0(\mathbf{z})+\mathcal{O}(\mathbf{z}^{2M+3})+
\varepsilon
\left(
\mathbf{w}_1^0(\phi)+\mathcal{O}(\mathbf{z})
\right)
+\mathcal{O}(\varepsilon^2)=
\mathbf{W}_0(\mathbf{z})+\mu^{2M+2} \mathbf{w}_1^0(\phi)+\mathcal{O}(\mu^{2M+3}),
\\
\mathbf{R}(\mathbf{z},\phi)&=\mathbf{R}_0(\mathbf{z})+\mathcal{O}(\mathbf{z}^{2M+3})+\varepsilon \left( \mathbf{r}_1^0(\phi)+ \mathcal{O}(\mathbf{z})\right) +\mathcal{O}(\varepsilon^2)=
\mathbf{R}_0(\mathbf{z})+ \mu^{2M+2} \mathbf{r}_1^0(\phi)+\mathcal{O}(\mu^{2M+3}).
\end{split}
\label{eq:rescaled_SSM}
\end{equation}

The first error term ($\mathcal{O}(\mathbf{z}^{2M+3})$) results from the order-$(2M\!+\!1)$ truncation of autonomous SSM (cf.~eq.~\eqref{eq:R0}). The last error term arises from the truncated expansion of the SSM at the order two in $\varepsilon$ in the expansion~\eqref{eq:expans_SSM+dyn}. The error term $\varepsilon\mathcal{O}(\mathbf{z})$ arises from the truncation of $\mathbf{W}_1(\mathbf{z},\phi)$ at the zeroth order in $\mathbf{z}$. Rewriting the invariance condition~\eqref{eq:expans_SSM+dyn} in terms of the new variable $\mu$, we obtain

\begin{equation}
\begin{split}
&\mathbf{A} 
(\mathbf{W}_0(\mathbf{z})+\mu^{2M+2} \mathbf{w}_1^0(\mathbf{\phi}))
+
\mathbf{G}_{nlin}(\mathbf{W}_0(\mathbf{z}))+
\mu^{2M+2} \mathbf{G}_{ext}(\mathbf{\phi})\\
&=
\frac{\partial \mathbf{W}_0(\mathbf{z}) }{\partial \mathbf{z}}\mathbf{R}_0(\mathbf{z})
+\mu^{2M+2}
\left[
\mathbf{w}^{(1,0)}_0,
\mathbf{w}^{(0,1)}_0
\right]\mathbf{R}_1(\mathbf{z},\mathbf{\phi})
+
\mu^{2M+2}\frac{\partial \mathbf{w}_1^0(\phi) }{\partial \mathbf{\phi}}\Omega
+\mathcal{O}(\mu^{2M+3}).
\end{split}
\label{eq:inv_cond_mu}
\end{equation} 

The time-varying terms in eq.~\eqref{eq:inv_cond_mu} are of order $\mathcal{O}(\mu^{2M+2})$. Since the autonomous dynamics $\mathbf{R}_0(\mathbf{z})$ (cf.~eq.~\eqref{eq:R0}) and the corresponding $\mathbf{W}_0(\mathbf{z})$ are of lower order in $\mu$, the formulas from Szalai~et~al.~\cite{article:SSM_SysID}, which are derived for the autonomous case, apply here as well. As those formulas for $\mathbf{R}_0$ are only applicable when the linear part of~\eqref{eq:system_ss} is diagonalized, we first perform a change of coordinates~\mbox{$\mathbf{w}_{1}^0(\phi)\!=\!\mathbf{V}\mathbf{v}(\phi)$} and left-multiply~\eqref{eq:inv_cond_mu}  $\mathbf{V}^{-1}$  to obtain

\begin{equation}
\mathcal{O}(\mu^{2M+2}): \qquad 
\mathbf{V}^{-1} \mathbf{A}\mathbf{V}\mathbf{v}
+\mathbf{V}^{-1}
\begin{bmatrix}
0\\
\mathbf{M}^{-1} \mathbf{f}
\end{bmatrix}
(e^{i \Omega t}+e^{-i \Omega t})
=
\dot{\mathbf{v}}(\phi t)+
\begin{bmatrix}
\delta_{1 l} & \delta_{1 (l+N)}\\ 
\delta_{2 l} & \delta_{2 (l+N)}\\
\vdots & \vdots\\
\delta_{(2N) l} & \delta_{2N (l+N)}
\end{bmatrix}
\mathbf{R}_1(\phi),
\label{eq:first_order_solve_diag}
\end{equation}

\noindent where the specific terms  for $\mathbf{w}^{(1,0)}_1$ and  $\mathbf{w}^{(0,1)}_1$ are taken from~Szalai~et~al.~\cite{article:SSM_SysID}. The matrix $\mathbf{V}^{-1} \mathbf{A}\mathbf{V}$ is diagonal, containing the eigenvalues $\lambda_j$. Furthermore, we have also inserted the canonical forcing~\eqref{eq:forcing} in the eq.~\eqref{eq:first_order_solve_diag}. For the $n$-th coordinate eq.~\eqref{eq:first_order_solve_diag}, gives

\begin{equation}
\lambda_n v_{n}
+p_{n}(e^{i \Omega t}+e^{-i \Omega t})
=
\dot{v}_{n}+\delta_{n l} r_c e^{i \Omega t}+\delta_{n (l+N)} \bar{r}_c e^{-i \Omega t},
\label{eq:first_order_solve_n}
\end{equation}

\noindent where the complex amplitude $p_n$ of the forcing is defined as

\begin{equation}
p_{n}= \mathbf{t}_n\begin{bmatrix}
0\\
\mathbf{M}^{-1}\mathbf{f}
\end{bmatrix}.
\label{eq:Red_dyn_forcing_general}
\end{equation}

We observe, that eq.~\eqref{eq:first_order_solve_n} is a linear ordinary differential equation for the unkown coefficients~$v_n$. Therefore, the periodic solutions of \eqref{eq:first_order_solve_n} can be obtained as in the previous section from the Dunhamel's principle. Due to the Kronecker delta in eq.~\eqref{eq:first_order_solve_n}, three cases arise:

\begin{subequations}
	\begin{align}
	v_{l}&=\frac{p_l-r}{i\Omega-\lambda_l}e^{i\Omega t}+\frac{p_l}{-i\Omega-\lambda_l}e^{-i\Omega t},
	\label{eq:modal_solve_l}
	\\
	v_{l+N}&=\frac{\overline{p}_l}{i\Omega-\overline{\lambda}_l}e^{i\Omega t}+\frac{\overline{p}_l-\overline{r}}{-i\Omega-\overline{\lambda}_l}e^{-i\Omega t},
	\label{eq:modal_solve_lN}
	\\
	v_{n}&=\frac{p_n}{i\Omega-\lambda_n}e^{i\Omega t}+\frac{p_n}{-i\Omega-\lambda_n}e^{-i\Omega t}, \qquad n \neq l,l+N.  
	\label{eq:modal_solve_rest}
	\end{align}
\end{subequations}

The amplitude $p_l$  coincides with $r$ (cf. eq.~\eqref{eq:Red_dyn_forcing} and \eqref{eq:Red_dyn_forcing_general}). Therefore, the first term in~\eqref{eq:modal_solve_l} and the second term in~\eqref{eq:modal_solve_lN} vanish. These are the terms, that would create small denominators in the parameterization of the SSM ($\mathbf{w}_1^0$). The choice of the coefficient $\mathbf{r}^0_1$ in~\eqref{eq:Red_dyn_forcing} eliminates these small denominators. Because we assume that the near-resonance conditions~\eqref{eq:ext_nearres_cond} are satisfied, small  denominators in~\eqref{eq:modal_solve_rest} do not arise, unless the forcing frequency is close to the imaginary part of another eigenvalue different from $\lambda_l$ and $\overline{\lambda}_l$. In this case however, the SSM needs to be constructed tangent to this specific subspace.  

Changing back from the diagonal form to the physical coordinates, we recover  $\mathbf{w}_1^0(\phi)$ from~\eqref{eq:SSM_nearres}.

\subsection{The choice of the eigenvectors}
\label{sec:eig_choice}
In the following, we show that the eigenvectors $\mathbf{v}_j$ can be normed such that the constant $r_c$ arising in the reduced dynamics~\eqref{eq:Red_dyn_forcing} is purely imaginary. First, note that for a general choice of eigenvectors, $r_c$ is complex, i.e., can be expressed as 
\begin{equation}
r_{c}=ire^{i\varphi}= \mathbf{t}_l\begin{bmatrix}
0\\
\mathbf{M}^{-1}\mathbf{f}
\end{bmatrix},	
\end{equation} 
and $r_c$ is not purely imaginary for $\varphi\! \neq\! \pm n \pi$. Multiplying  $\mathbf{V}$ with $e^{i\varphi}$ and realizing that the vector $\mathbf{t}_l$ is the $l$-th row of $\mathbf{V}^{-1}$ (cf.~eq.~\eqref{eq:inverse_mat}), we obtain 
\begin{equation}
e^{-i\varphi}\mathbf{t}_l\begin{bmatrix}
0\\
\mathbf{M}^{-1}\mathbf{f}
\end{bmatrix}=e^{-i\varphi}ire^{i\varphi}=ir,	
\end{equation} 
and $r_c$ is purely imaginary holds for $\widetilde{\mathbf{V}}=e^{i\varphi}\mathbf{V}$. This proofs Corollary~\ref{cor:eig_vecs}. 

Following the proof of Corollary~\ref{cor:eig_vecs}, we show that in case of purely symmetric system martices ($\mathbf{N}=\mathbf{0}$ and $\mathbf{G}=\mathbf{0}$) and structural damping, the constant $r_c$ is purely imaginary, if we mass normalize the mode shapes $\mathbf{e}_j$. With the real mode shape matrix $E$ and the notation~\eqref{eq:mode_shp}, the matrix of eigenvectors $\mathbf{V}$ is given by 
\begin{equation}
\mathbf{V}=
\begin{bmatrix}
\mathbf{E} & \mathbf{E} \\
\mathbf{E} \mathbf{\Lambda} & \mathbf{E} \overline{\mathbf{\Lambda}}
\end{bmatrix}.
\label{eq:V_structural_damp}
\end{equation}
To compute the constant $r_c$ explicitly, we need to compute the inverse of the martix $\mathbf{V}$~(cf.~eq.~\eqref{eq:Red_dyn_forcing}).  By blockwise inversion of~\eqref{eq:V_structural_damp}, we obtain
\begin{equation}
\mathbf{V}^{-1}=
\begin{bmatrix}
\mathbf{E}^{-1}+ (\overline{\mathbf{\Lambda}}-\mathbf{\Lambda})^{-1}\mathbf{E}^{-1}\mathbf{\Lambda}\mathbf{E}^{-1} & -(\overline{\mathbf{\Lambda}}-\mathbf{\Lambda})^{-1}\mathbf{E}^{-1}\\
(\overline{\mathbf{\Lambda}}-\mathbf{\Lambda})^{-1}\mathbf{E}^{-1}\mathbf{\Lambda}\mathbf{E}^{-1}& (\overline{\mathbf{\Lambda}}-\mathbf{\Lambda})^{-1}\mathbf{E}^{-1}\\
\end{bmatrix}.
\label{eq:inv_Vss}
\end{equation}
Since the mode-shape-matrix $\mathbf{E}$ and its inverse are real,  the last $N$ columns of~\eqref{eq:inv_Vss} are purely imaginary. Therefore, a multiplication by the forcing in eq.~\eqref{eq:Red_dyn_forcing} will always lead to a purely imaginary $r_c$.

\subsection{Amplitude, phase shift and stability of the $T$-periodic orbits of the reduced dynamics}
\label{app:der_res3}

In the following, we compute periodic orbits with the same period as the forcing of the reduced dynamics~\eqref{eq:Red_dyn_forcing}. These orbits will determine the steady-state response of the system~\eqref{eq:sys_phys}. First, we transform the reduced dynamics~\eqref{eq:dynamics_R_nearres} into the polar coordinates~\eqref{eq:pol_trans}, which yields 
\begin{equation}
\begin{split}
\dot{\rho}&=Re(\lambda_l)\rho+\sum_{m=1}^M Re(\beta_m)\rho^{2m+1}+
\varepsilon 
r \sin(\theta-\phi),
\\
\dot{\theta}&=Im(\lambda_l)+\sum_{m=1}^M \mbox{Im}(\beta_m)\rho^{2m}+
\frac{\varepsilon}{\rho} 
\left( 
r\cos(\theta-\phi)
\right),
\\
\dot{\phi}&=\Omega.
\end{split}
\label{eq:dynamics_poolar}
\end{equation}
With the change of coordinates $\psi\!=\!\theta\!-\!\phi$, the dynamics \eqref{eq:dynamics_poolar} can be rewritten as
\begin{align}
\dot{\rho}&=\mbox{Re}(\lambda_l)\rho+\sum_{m=1}^M \mbox{Re}(\beta_m)\rho^{2m+1}
+\varepsilon 
r \sin(\psi) 
=
a(\rho)+
\varepsilon r \sin(\psi),
\label{eq:rho_eq}\\
\dot{\psi}&=\mbox{Im}(\lambda_l)+\sum_{m=1}^M \mbox{Im}(\beta_m)\rho^{2m}+\frac{\varepsilon}{\rho} 
\left( 
r\cos(\psi)
\right)-\Omega
=b(\rho)+\frac{\varepsilon}{\rho} 
\left( 
r\cos(\psi)
\right)-\Omega,
\label{eq:phase_shift}
\\
\dot{\phi}&=\Omega.
\end{align}
The angle $\psi$ represents the phase shift between the forcing and the system response. At the steady state of~\eqref{eq:system_ss}, the amplitude $\rho$, as well as the phase shift $\psi$ are constant. The trigonometric functions in \eqref{eq:rho_eq} and \eqref{eq:phase_shift} can be eliminated by solving \eqref{eq:rho_eq} for $\sin(\psi)$ and \eqref{eq:phase_shift} for $\cos(\psi)$ and adding the square of both equations. Thereby we obtain eq.~\eqref{eq:ploy2solvem}. We determine the phase relation~\eqref{eq:phase_ss}  by solving the steady state response in of eq.~\eqref{eq:phase_shift} for $\psi$. The stability of such solutions can be obtained by evaluating the eigenvalues of the Jacobian of \eqref{eq:rho_eq} and \eqref{eq:phase_shift} with respect to $\rho$ and $\psi$, which we state in~\eqref{eq:jacobian}.

\section{Coefficients for the autonomous SSM and the reduced dynamics}
\label{app:Coeff_Szalai}
We recall here for completeness the parameterization of the autonomous SSM ($\mathbf{W}_0(\mathbf{z})$) and the reduced dynamics ($\mathbf{R}_0(\mathbf{z})$) from Szalai~et~al.~\cite{article:SSM_SysID}. To diagonalize the linear part, we apply the transformation
\begin{equation}
\mathbf{y}= \mathbf{V}\mathbf{x}
\end{equation}
to the  autonomous limit ($\varepsilon\!\rightarrow\!0$) of system~\eqref{eq:system_ss} and obtain
\begin{equation}
\dot{\mathbf{y}}
=	
\mathbf{V}^{-1}\mathbf{A}  \mathbf{V} \mathbf{y}
+
\mathbf{V}^{-1}\mathbf{G}_{nlin}( \mathbf{V}\mathbf{y})
=
\mathbf{ \Lambda}\mathbf{y}+\mathbf{G}(\mathbf{y}).
\label{eq:diag_sys}
\end{equation}
The Taylor series of the $j$-th entry of nonlinear terms $\mathbf{G}$ is
\begin{equation}
G_j=\sum_{\mathbf{m}\in \mathbb{N}^{2N}_0}g_j^{\mathbf{m}}\mathbf{y}^{\mathbf{m}}.
\end{equation}
As Szalai~et~al.~\cite{article:SSM_SysID}, we use $(p@j)$ to denote an integer multi-index whose elements are zero, except for the index at the $j$-th position, which is equal to p, i.e.
\begin{equation}
(p@j):=\left(0,...,\underset{j-1}{0},\underset{j}{p},\underset{j+1}{0},....,0\right) \in \mathbb{N}^{2N} .
\end{equation}
We also use this notation to refer to multi-indices with multiple entries $(p@j_1,q@j_2)$ and in case of $j_1\!=\!j_2$ the corresponding entry is $p+q$, i.e.
\begin{equation*}
\begin{split}
(p@j_1,q@j_2)&:=\left(0,...,\underset{j_1-1}{0},\underset{j_1}{p},\underset{j_1+1}{0},....,
,\underset{j_2-1}{0},\underset{j_2}{q},\underset{j_2+1}{0},...,0\right), \\
 (p@j,q@j)&:=\left(0,...,\underset{j-1}{0},\underset{j}{p+q},\underset{j+1}{0},...,0\right).
\end{split}
\end{equation*}
With this notation the coefficients of the parameterization $\mathbf{W}_0(\mathbf{z})$ for $j=1,...,2N$ are given by

\[\arraycolsep=1.4pt\def\arraystretch{2.7}
\begin{array}{c}
\displaystyle
w_j^{(1,0)}=\delta_{jl}, \qquad w_j^{(1,0)}=\delta_{j(l+N)},
\\
\displaystyle
w_j^{(2,0)}=\frac{g_j^{(2@l)}}{2\lambda_l-\lambda_j},
\qquad 
w_j^{(1,1)}=\frac{g_j^{(1@l,1@(l+N))}}{\lambda_l+\overline{\lambda}_{l}-\lambda_j},
\qquad
w_j^{(0,2)}=\frac{g_j^{(2@(l+N))}}{2\overline{\lambda}_{l}-\lambda_j},
\\
\displaystyle
w_j^{(3,0)}=\frac{\sum_{q=1}^{2N}(1+\delta_{lq})g_j^{(1@l,1@q)}w_q^{(2,0)}+g_j^{(3@l)}}{3\lambda_l-\lambda_j},
\\
w_j^{(0,3)}=\frac{\sum_{q=1}^{2N}(1+\delta_{(l+N)q})g_j^{(1@(l+N),1@q)}w_q^{(0,2)}+g_j^{(3@(l+N))}}{3\overline{\lambda}_{l}-\lambda_j},
\\
\displaystyle
w_j^{(2,1)}=
\frac{\sum_{q=1}^{2N}(1+\delta_{lq})g_j^{(1@l,1@q)}w_q^{(1,1)}
	+\sum_{q=1}^{2N}(1+\delta_{(l+N)q})g_j^{(1@(l+N),1@q)}w_q^{(2,0)}
	+g_j^{(3@l)}}
{2\lambda_l+\overline{\lambda}_{l}-\lambda_j},
\\
\displaystyle
w_j^{(1,2)}=
\frac{\sum_{q=1}^{2N}(1+\delta_{lq})g_j^{(1@l,1@q)}w_q^{(0,2)}
	+\sum_{q=1}^{2N}(1+\delta_{(l+N)q})g_j^{(1@(l+N),1@q)}w_q^{(1,1)}
	+g_j^{(3@l)}}
{\lambda_l+2\overline{\lambda}_{l}-\lambda_j}.
\end{array}
\]

\noindent The coefficient $\beta_1$ of the reduce dynamics~\eqref{eq:R0} is given by
\begin{equation}
\beta_1=\sum_{q=1}^{2N}(1+\delta_{lq})g_j^{(1@l,1@q)}w_q^{(1,1)}
+\sum_{q=1}^{2N}(1+\delta_{(l+N)q})g_j^{(1@(l+N),1@q)}w_q^{(2,0)}
+g_j^{(3@l)}.
\label{eq:beta}
\end{equation}
To compute the $\mathcal{O}(5)$ autonomous SSM ($M\!=\!2$), we provide a \textsc{Matlab} script as electronic supplementary material.

\bibliographystyle{abbrv}	
\bibliography{Backbone_SSM_Lit}

\end{document}